\pdfoutput=1
\documentclass{shinyart}

\usepackage[utf8]{inputenc}
\usepackage{enumitem} 
\setlist[enumerate]{label={(\roman*)}}
\usepackage[compress]{cite}

\newcommand{\Z}{\mathbb{Z}}
\newcommand{\N}{\mathbb{N}}
\newcommand{\R}{\mathbb{R}}
\newcommand{\Rbar}{\overline{\R}}
\newcommand{\calL}{\mathcal{L}}

\renewcommand{\epsilon}{\varepsilon}
\newcommand{\norm}[1]{\|#1\|}
\newcommand{\abs}[1]{|#1|}
\newcommand{\bigabs}[1]{\left|\,#1\,\right|}
\newcommand{\dualpair}[1]{\langle#1\rangle}
\newcommand{\di}{\mathrm{d}}
\newcommand{\Prob}[1]{\mathbb{P}\left[#1\right]}
\newcommand{\Exp}[1]{\mathbb{E}\left[#1\right]}

\DeclareMathOperator*{\argmin}{arg\,min}
\DeclareMathOperator*{\argmax}{arg\,max}
\DeclareMathOperator{\spn}{span}

\DeclareMathOperator{\proj}{proj}

\newcommand{\Eg}{{E_\gamma}}
\newcommand{\Egd}{{E_\gamma^\delta}}
\newcommand{\Egn}{{E_\gamma^0}}
\newcommand{\xmap}{x_{\mathrm{MAP}}}

\newtheorem{assumption}[theorem]{Assumption}

\usepackage{pgfplots}
\pgfplotsset{compat=newest}

\begin{document}

\title{Generalized modes in Bayesian inverse problems}
\author{%
    Christian Clason\thanks{Faculty of Mathematics, University Duisburg-Essen, Thea-Leymann-Strasse 9, 45127 Essen, Germany\newline%
    (\email{christian.clason@uni-due.de}, \email{remo.kretschmann@uni-due.de})}
    \and Tapio Helin\thanks{School of Engineering Science, LUT University, P.O.~box 20, FI-53851 Lappeenranta, Finland (\email{tapio.helin@lut.fi}), \emph{previously} Department of Mathematics and Statistics, University of Helsinki, Finland}
    \and Remo Kretschmann\footnotemark[1]
    \and Petteri Piiroinen\thanks{Department of Mathematics and Statistics, University of Helsinki, Gustaf Hallströmin katu 2b, FI-00014 Helsinki, Finland (\email{petteri.piiroinen@helsinki.fi})}
}
\hypersetup{
    pdftitle = {Generalized modes in Bayesian inverse problems},
    pdfauthor = {C. Clason, T. Helin, R. Kretschmann, P. Piiroinen},
}

\maketitle

\begin{abstract}
    Uncertainty quantification requires efficient summarization of high- or even infinite-dimensional (i.e., non-parametric) distributions based on, e.g., suitable point estimates (modes) for posterior distributions arising from model-specific prior distributions.
    In this work, we consider non-parametric modes and MAP estimates for priors that do not admit continuous densities, for which previous approaches based on small ball probabilities fail. We propose a novel definition of generalized modes based on the concept of approximating sequences, which reduce to the classical mode in certain situations that include Gaussian priors but also exist for a more general class of priors. The latter includes the case of priors that impose strict bounds on the admissible parameters and in particular of uniform priors. For uniform priors defined by random series with uniformly distributed coefficients, we show that generalized MAP estimates -- but not classical MAP estimates -- can be characterized as minimizers of a suitable functional that plays the role of a generalized Onsager--Machlup functional. This is then used to show consistency of nonlinear Bayesian inverse problems with uniform priors and Gaussian noise.
\end{abstract}

\section{Introduction}

Uncertainty quantification is concerned with the effects of random perturbations -- characterized by a \emph{prior distribution} -- on mathematical models of real-world situations such as solutions of partial differential equations or on parameter estimation problems (the latter being referred to as \emph{Bayesian inverse problems}). The full goal is a characterization of the \emph{posterior distribution} of the solution as a function of the random perturbation; however, this is in general infeasible due to the high- or even infinite-dimensionality of the posterior. It is therefore necessary to summarize the distribution using, e.g., point estimates. One possibility is the conditional mean, defined as the expected value of the posterior, which involves computing a high-dimensional integral using, e.g., Markov Chain Monte Carlo (MCMC) algorithms. A popular alternative is the \emph{maximum a posteriori (MAP) estimate}, defined as the point maximizing -- in an appropriate sense -- the posterior distribution.
This alternative is attractive due to the fact that its variational characterization in many cases allows its efficient computation as the solution of an optimization problem. In some finite-dimensional settings, it can furthermore be justified from Bayesian decision theory \cite{burgerlucka2014,pereyra2016}. It can also be used to compute Laplace approximations of the posterior density \cite{tierney1986accurate}, which have been used successfully to counteract numerical instabilities that often occur due to the concentration effect of the posterior for highly informative data or small noise; see \cite{schillingsschwab2016, chen2017hessian} and the references therein. In general, the acceleration of existing algorithms by approximations based on the MAP estimate or on second-order information such as the Hessian or Fisher information matrix is an active field of research \cite{petramartinstadlerghattas2014, beckdiaespathquantempone2018}.

Although the underlying model (e.g., a partial differential equation) is often infinite-dimension\-al, numerical computations of course have to rely on finite-dimensional approximations. However, there is a fundamental need to understand whether the problem scales well with increasing dimensions required by increasing accuracy demands and whether the related estimates have well-behaving limits. Such robustness is provided by studying the estimation problem and deriving algorithms for the infinite-dimensional setting that then apply to any (conforming) discretization. However, this is much more involved than in finite dimensions, since the posterior no longer has a natural density representation which significantly complicates the definition and study of the underlying conditional probabilities. In particular, it is not clear in the infinite-dimensional setting whether a MAP estimate exists and whether it has a variational characterization that allows efficient computation. Such results have so far only been obtained under strong assumptions on the prior distribution. Removing or relaxing these assumptions would be of practical relevance since the prior models any information on the uncertainty entering the problem and its correct choice is therefore of fundamental importance in uncertainty quantification. In particular, infinite-dimensional priors that allow for strict bounds on the unknown parameters have so far not been studied.

The goal of this work is to address these issues specifically in the context of Bayesian inverse problems. 
To motivate our contribution, we recall that a Bayesian inverse problem is concerned with quantifying the uncertainty in an unknown $x$ given some prior distribution $\mu_0$ of the probable values of $x$ together with a (possibly indirect) noisy observation $y = y^\delta$ subject to random perturbations; see, e.g., the early work \cite{JF70}, the books \cite{AT05,KS06}, and the more recent works \cite{Lasanen1,Stuart,Dashti2017} as well as the references therein.
It is well-known, e.g., from \cite{Stuart} that under quite general conditions, the posterior distribution $\mu_{\text{post}}$ is absolutely continuous with respect to the prior and is given by the Bayes' formula
\begin{equation}
    \label{eq:post_dens}
    \frac{\di\mu_{\text{post}}}{\di\mu_0}(x|y) = \exp(-\Phi(x; y)),
\end{equation}
where $\Phi$ is the (scaled) negative log-likelihood. 
As in any uncertainty quantification problem, the efficient summary of the information contained in $\mu_{\text{post}}$ is challenging in high dimensions, even if a more explicit characterization is available. In principle, MCMC algorithms can be applied to approximate the conditional mean using \eqref{eq:post_dens}, but the computational effort required by MCMC algorithms in large-scale problems is significant. Correspondingly, there has been interest in how approximate Bayesian methods (see, e.g., \cite{cotter2010approximation, LSW17a, schwab2012sparse, MarzoukNajm2009, marzouk2007stochastic}) can be used 
effectively in inverse problems. As mentioned above, a central tool in this research effort is the MAP estimate, which can be formulated in a finite-dimensional setting as the solution to the optimization problem
\begin{equation}
    \label{eq:var_map}
    \xmap = \argmin_x \left\{\Phi(x; y) + R(x)\right\},
\end{equation}
where $R$ is the negative log-prior density. 
The study of infinite-dimensional or \emph{non-parametric} MAP estimates -- or more generally, non-parametric \emph{modes} -- was only recently initiated by Dashti et al. in \cite{dashti2013map} in the context of nonlinear Bayesian inverse problems with Gaussian prior and noise distribution. Due to the lack of a Lebesgue measure on infinite-dimensional spaces and, with it, a natural density of the posterior distribution, the definition of a mode had to be generalized.
Rather than looking for maxima of the density function, one considers points for which small balls around this point have asymptotically maximal probability as their radius tends to zero. Since the maximal probability is also defined with respect to the distribution, this definition avoids the need for densities; see \cref{def:mode} below.
In the Gaussian case, one can then show an explicit relation between the objective functional in \eqref{eq:var_map} and the \emph{Onsager--Machlup functional} defined via the limit of small ball probabilities. This can be used to give a statistical justification for the objective functional that is consistent with the finite-dimensional definition \cite{dashti2013map}.  
More recently, there has been a series of results \cite{helin2015maximum, DS16, agapiou2018sparsity} that extend the definition and scope of non-parametric MAP estimation. For example, weak MAP estimates were proposed and studied in \cite{helin2015maximum} in the context of linear Bayesian inverse problems with a general class of priors. The authors used the tools from the differentiation and quasi-invariance theory of measures (developed by Fomin and Skorokhod and discussed in detail in \cite{Boga10}) to connect the zero points of the logarithmic derivative of a measure to the minimizers of the Onsager--Machlup functional. 

This program of generalization is motivated by the fact that in inverse problems, the prior is not only important for posterior modeling but also plays a key role in the successful stabilization of the inherent ill-posedness of the problem \cite{Stuart}. 
The prior thus needs to be carefully designed to reflect the best possible subjective information available; here we only mention priors that reflect the sparsity \cite{lassas2009discretization, kolehmainen2012sparsity}, hierarchical structure \cite{wang2004hierarchical, calvetti2008hypermodels}, or anisotropic features \cite{kaipio1999inverse} of the unknown.
However, in the non-parametric case, all results regarding MAP estimates known to the authors require continuity of the prior; but even simple one-dimensional measures with discontinuous density can fail to have a mode in terms of the definition given in \cite{dashti2013map}; see \cref{ex:standard} below. 
This is limiting, since suitable prior modeling may involve imposing strict bounds on the admissible values of $x$ emerging from some fundamental properties of the application. As a simple example, consider the classical inverse problem of X-ray medical imaging \cite{mueller2012linear}. Since it is reasonable to assume that there are no radiation sources inside the patient, the attenuation of X-rays is positive throughout the body. 
A related situation occurs in electrical impedance tomography -- and more generally, in parameter estimation problems for partial differential equations -- where pointwise upper and lower bounds need to be imposed on the parameter to ensure well-posedness of the forward problem; reasonable bounds are often available from a priori information on, e.g., the kinds of tissue or material expected in the region of interest. 
Furthermore, it is known in the deterministic theory that restriction of the unknown to a compact set can stabilize an inverse problem without additional (possibly undesired) regularization terms; this is sometimes referred to as \emph{quasi-solution} or \emph{Ivanov regularization} \cite{Iva62,IvaVasTan02,LorWor13,NeuRam14}, and the use of pointwise bounds for this purpose has recently been studied in \cite{KK17,CK18}. 
In the Bayesian approach, this would correspond to priors with densities whose mass is contained in a compact set.
Together, this motivates the study of imposing such hard bounds as part of the prior in Bayesian inverse problems.

The main contribution of our work is therefore two-fold: First, we introduce a novel definition of \emph{generalized modes} for probability measures and characterize conditions under which our definition coincides with the previous definition of modes given in \cite{dashti2013map} (called \emph{strong modes} in the following). 
More precisely, we show that if the Radon--Nikodym derivative $d\mu(\cdot - h)/d\mu$ of the translated measure $\mu$ has certain equicontinuity properties (described in \cref{subsec:gen_properties}) over
a dense set of translations $h$, any generalized mode is also a strong mode.
In particular, we demonstrate that our definition for generalized modes does not introduce pathological modes in the case of continuous densities and that strong and generalized modes coincide for Gaussian measures.
Second, we consider \emph{uniform priors} defined via the random series
\begin{equation}
    \label{eq:uniform}
    \xi = \sum_{k=1}^\infty \gamma_k \xi_k \phi_k,
\end{equation}
where $\xi_k \sim \mathcal{U}[-1,1]$ are uniformly distributed and $\gamma_k$ are suitable weights; such priors clearly have a discontinuous (if any) density.
We further show that the uniform prior \eqref{eq:uniform} does not have any strong mode that touches the bounds (i.e., where $\xi_k\in\{-1,1\}$ for some $k\in\N$) and that this implies that the posterior also does not have a strong mode in general. However, we prove that such points are in fact generalized modes of the prior. 
Furthermore, we show that for the uniform prior defined via \eqref{eq:uniform}, the corresponding generalized MAP estimates of \eqref{eq:post_dens} can be characterized in a natural manner as minimizers of \eqref{eq:var_map}. We also provide a weak consistency result regarding the generalized MAP estimates in line of the previous work \cite{dashti2013map, agapiou2018sparsity}. 

\bigskip

The paper is organized as follows.
In \cref{sec:modes}, we give our definition of generalized modes and illustrate the definition for the examples of a measure with discontinuous density (which does not admit a strong mode) and of Gaussian measures (where generalized and strong modes coincide). A general investigation of conditions for the generalized and strong modes to coincide is carried out in \cref{sec:relation}. We next construct uniform priors on an infinite-dimensional Banach space and characterize its generalized modes in \cref{sec:prior}. In \cref{sec:MAP}, we study Bayesian inverse problems with such priors and derive a variational characterization of generalized modes that plays the role of a generalized Onsager--Machlup functional. This is used in \cref{sec:consistency} to show consistency of nonlinear Bayesian inverse problems with uniform priors and Gaussian noise.

\section{Generalized modes}\label{sec:modes}

Let $X$ be a separable Banach space and $\mu$ a probability measure on $X$.
Throughout, let $B^\delta(x) \subset X$ denote the open ball around $x \in X$ with radius $\delta$ and let
\begin{equation*}
    M^\delta := \sup_{x\in X} \mu(B^\delta(x)) 
\end{equation*}
for each $\delta > 0$ denote the maximal probability of a ball of radius $\delta$ under $\mu$. We first recall the definition of a mode introduced in \cite{dashti2013map}.
\begin{definition} \label{def:mode}
    A point $\hat{x} \in X$ is called a \emph{(strong) mode} of $\mu$ if
    \begin{equation*} 
        \lim_{\delta \to 0} \frac{\mu(B^\delta(\hat{x}))}{M^\delta} = 1. 
    \end{equation*}
    The modes of the posterior measure $\mu^y$ are called \emph{maximum a posteriori (MAP) estimates}.
\end{definition}
This definition compares the rate in which the probability of a small ball around the point $\hat{x}$ decreases to the 
rate achieved by choosing every ball to maximize its probability, which is the lowest rate than can be achieved.
If these rates agree asymptotically, then $\hat{x}$ is a mode. Intuitively, a mode maximizes the probability distribution in the sense that asymptotically, balls of a fixed radius around it contain maximal probability.
Compared to the classical definition, it avoids reference to both point evaluations and densities, which makes it appropriate in the infinite-dimensional setting.
Also note that this is a global definition and hence disregards local modes.

It is straightforward to construct a probability distribution with discontinuous Lebesgue density which does not have a mode according to \cref{def:mode}.
\begin{example} \label{ex:standard}
    Let $\mu$ be a probability measure on $\R$ with density $p$ with respect to the Lebesgue measure, defined via
    \begin{equation*}
        \tilde p(x) := \begin{cases}
            1 - x & \text{if }x \in [0,1], \\
            0 & \text{otherwise},
        \end{cases}
        \qquad
        p(x) := \frac{\tilde p(x)}{\int_{\R} \tilde p(x)\,dx}.
    \end{equation*}
    Clearly, $\hat x = 0$ maximizes $p$; however, it is not a strong mode.
    First, note that for every $\delta > 0$,
    \begin{equation*} 
        \mu(B^\delta(\delta)) = \sup_{x\in\R} \mu(B^\delta(x)) = M^\delta. 
    \end{equation*}
    Hence of all balls with radius $\delta$, the one around $x^\delta := \delta$ has the highest probability.
    However, although $x^\delta\to \hat{x}$ we have for $\delta$ small enough that
    \begin{equation*}
        \lim_{\delta \to 0}  \frac{\mu(B^\delta(0))}{\mu(B^\delta(\delta))}
        =\lim_{\delta \to 0}  \frac{\delta(1 - \frac12\delta)}{2\delta(1 - \delta)} = \frac 12 < 1,
    \end{equation*}
    and thus \cref{def:mode} is not satisfied.
\end{example}
The above example illustrates the problem with classical modes for discontinuous densities: Since the point $\hat x$ lies at the discontinuity of the density, \emph{any} ball of radius $\delta$ around $\hat{x}$ has a mass of at most $\delta$ (as opposed to a maximal possible mass of roughly $2\delta$), and this loss of mass is conserved in the limit.
We do, however, have a family $\{B^\delta(\delta)\}_{\delta > 0}$ of balls that each have maximal probability and whose center points converge toward $\hat{x}$.
This gives rise to the idea of replacing the fixed center point $\hat{z}$ in the definition of a mode by an ``approximating sequence'' $\{w_\delta\}_{\delta > 0}$ that converges to $\hat{x}$ as $\delta \to 0$.
(Similar limiting arguments also serve as the basis of constructions of generalized derivatives for non-differentiable functions such as Clarke's generalized directional derivative or Mordukhovich's limiting subdifferential.)

We are thus lead to the following definition.
\begin{definition} \label{def:gmode}
    A point $\hat x \in X$ is called a \emph{generalized mode} of $\mu$ if for every sequence $\{\delta_n\}_{n\in\N} \subset (0,\infty)$ with $\delta_n \to 0$ there exists an \emph{approximating sequence} $\{w_n\}_{n\in\N} \subset X$ with $w_n \to \hat x$ in $X$ and
    \begin{equation}
        \label{eq:gmode_conv}
        \lim_{n \to \infty} \frac{\mu(B^{\delta_n}(w_n))}{M^{\delta_n}} = 1.
    \end{equation}
    We call generalized modes of the posterior measure $\mu^y$ \emph{generalized MAP estimates}.
\end{definition}
Note that by \cref{def:gmode}, every strong mode $\hat x \in X$ is also a generalized mode with the approximating sequence $w_n := \hat x$.
Also note that it is not necessary for the balls $B^{\delta_n}(w_n)$ to each have maximal probability; their probabilities only need to have the same asymptotic behavior as the maximal ball probabilities $M^{\delta_n}$.

In \cref{ex:standard}, $\hat x = 0$ is a generalized mode with the approximating sequence $w_n := \delta_n$ for any positive sequence $\{\delta_n\}_{n\in\N}$ with $\delta_n \to 0$. In fact, we can use the following stronger condition to show that a point is a generalized mode.
\begin{lemma} \label{prop:curvecond}
    Let $\hat{x} \in X$. If there is a family $\{w^\delta\}_{\delta > 0} \subset X$ such that $w^\delta \to \hat{x}$ in $X$ as $\delta \to 0$ and
    \begin{equation*}
        \lim_{\delta \to 0} \frac{\mu(B^\delta(w^\delta))}{M^\delta} = 1,
    \end{equation*}
    then $\hat{x}$ is a generalized mode of $\mu$.
\end{lemma}
In \cref{ex:standard}, this condition is obviously satisfied for $\hat x = 0$ and $w^\delta := \delta$.

\begin{remark}
    \Cref{def:gmode} can be further generalized by allowing for weakly converging approximating sequences. Conversely, we can restrict \cref{def:gmode} by coupling the convergence of $w_n$ to that of $\delta_n$ (cf.~\cref{prop:curvecond}). As we will show below, our definition has the advantage of not introducing pathological modes in the case of continuous densities (such as Gaussian or Besov distributions) while for the important case of a uniform prior, it introduces modes that are natural in terms of the variational characterization \eqref{eq:var_map}. 
\end{remark}

We now illustrate some of the key ideas of the generalized mode by considering the case of a Gaussian measure $\mu$ on $X$.
We assume that $\mu$ is centered and note that the results below trivially generalize to the non-centered case.
We will require the following quantitative estimate for Gaussian ball probabilities.
\begin{lemma}[\protect{\cite[Lem.~3.6]{dashti2013map}}]\label{lem:qualest}
    Let $x\in X$ and $\delta > 0$ be given. Then there exists a constant $a_1 > 0$ independent of $x$ and $\delta$ such that
    \begin{equation*}
        \frac{\mu(B^\delta(x))}{\mu(B^\delta(0))} \le e^{\frac{a_1}{2}\delta^2}e^{-\frac{a_1}{2}\left(\norm{x}_X - \delta\right)^2}.
    \end{equation*}
\end{lemma}
Now we can show that for a centered Gaussian measure $\mu$, \cref{def:mode,def:gmode} do indeed coincide.
\begin{theorem}
    \label{thm:gaussian_0}
    The origin $0 \in X$ is both a strong mode and a generalized mode of $\mu$, whereas all $x \in X \setminus \{0\}$ are neither a strong mode nor a generalized mode.
\end{theorem}
\begin{proof}
    First, by Anderson's inequality (see, e.g., \cite[Thm.~2.8.10]{Bogachev:1998}) we have that
    \begin{equation*}
        \mu(B^\delta(x)) \le \mu(B^\delta(0))
    \end{equation*}
    for all $x \in X$ and $\delta > 0$ and hence that
    $M^\delta = \mu(B^\delta(0))$ for all $\delta > 0$. 
    This immediately yields that $\hat x = 0$ is a strong mode and hence also a generalized mode for $\mu$.

    Now assume that $x \in X \setminus \{0\}$ is a generalized mode of $\mu$. Let $\{\delta_n\}_{n\in\N} \subset (0,\infty)$ with $\delta \to 0$ be arbitrary and let $\{w_n\}_{n\in\N} \subset X$ be the respective approximating sequence.
    Then, \cref{lem:qualest} yields
    \begin{equation*} 
        \frac{\mu(B^\delta(w))}{\mu(B^\delta(0))} 
        \le e^{-\frac{a_1}{2}\norm{w}_X(\norm{w}_X - 2\delta)} 
        \le e^{-\frac{a_1}{2}\left(\frac34\norm{x}_X\right)\left(\frac14\norm{x}_X\right)}
        = e^{-\frac{3a_1}{32}\norm{x}_X^2} =: A < 1 
    \end{equation*}
    for all $\delta \in (0, \delta_0]$ and $w \in B^{\delta_0}(x)$ with $\delta_0 := \frac14\norm{x}_X$.
    We choose $n_0 \in \N$ large enough such that $\norm{w_n - x}_X < \delta_0$ and $\delta_n < \delta_0$ for all $n \ge n_0$.
    Now taking the limit in the above yields
    \begin{equation*}
        \lim_{n\to\infty} \frac{\mu(B^{\delta_n}(w_n))}{M^{\delta_n}} 
        = \lim_{n\to\infty} \frac{\mu(B^{\delta_n}(w_n))}{\mu(B^{\delta_n}(0))}
        \le A < 1,
    \end{equation*}
    which is a contradiction. Hence, $x$ is not a generalized mode and thus cannot be a strong mode, either.
\end{proof}

The explicit bound in \cref{lem:qualest} plays an important role in \cref{thm:gaussian_0}. Compare this to the alternative approach using the \emph{Onsager--Machlup functional} $I(x)=\frac12\norm{x}_E^2$, defined as satisfying
\begin{equation*}
    \lim_{\delta \to 0} \frac{\mu(B^\delta(x_1))}{\mu(B^\delta(x_2))} = \exp\left( \frac12\norm{x_2}_E^2 - \frac12\norm{x_1}_E^2 \right),
\end{equation*}
which holds for all $x_1, x_2$ from the Cameron--Martin space $E \subset X$ of $\mu$ by \cite[Prop.~18.3]{Lifsic:1995}. Although this relation can be used to show that $0$ is the only generalized MAP estimate in $E$, it does not yield any information regarding $X \setminus E$. In the next section, we will for general measures $\mu$ -- under additional continuity assumptions -- extend results from a dense space such as $E$ to the whole space $X$.

\section{Relation between generalized and strong modes}\label{sec:relation}

In this section, we derive conditions under which our definition of generalized modes coincides with the standard notion of strong modes. We do this by based on further characterizations of the convergence of the approximating sequence in the definition of generalized modes.

\subsection{Characterization by rate of convergence}
\label{subsec:gen_properties}

Consider a general probability measure $\mu$ on $X$.
Let us first make the fundamental observation that for $\delta_n > 0$ and $\hat{x}, w_n \in X$, we have that
\begin{equation}
    \frac{\mu(B^{\delta_n}(\hat x))}{\sup_{x \in X} \mu(B^{\delta_n}(x))}
    = \frac{\mu(B^{\delta_n}(\hat x))}{\mu(B^{\delta_n}(w_n))} \cdot \frac{\mu(B^{\delta_n}(w_n))}{\sup_{x \in X} \mu(B^{\delta_n}(x))}.
    \label{eq:fundobs}
\end{equation}
Clearly, if $\hat x$ is a generalized mode, we have control over the right-most ratio in \eqref{eq:fundobs}. On the other hand, convergence of the ratio on the left-hand side in \eqref{eq:fundobs} is related to the definition of a strong mode. This leads to the following equivalence.
\begin{theorem} \label{thm:limit_condition}
    Let $\hat x \in X$ be a generalized mode of $\mu$. 
    Then $\hat x$ is a strong mode if and only if for every sequence $\{\delta_n\}_{n\in\N} \subset (0,\infty)$ with $\delta_n \to 0$, there exists an approximating sequence $\{w_n\}_{n\in\N} \subset X$ with $w_n \to \hat x$ and
    \begin{equation}
        \label{eq:property}
        \lim_{n \to \infty} \frac{\mu(B^{\delta_n}(\hat x))}{\mu(B^{\delta_n}(w_n))} = 1.
    \end{equation}
\end{theorem}
\begin{proof}
    Let $\{\delta_n\}_{n\in\N}$ be a positive sequence with $\delta_n \to 0$ and let $\{w_n\}_{n\in\N}$ be the corresponding approximating sequence that satisfies \eqref{eq:property}.
    Then
    \begin{equation*}
        \lim_{n\to\infty} \frac{\mu(B^{\delta_n}(\hat x))}
        {\sup_{x \in X} \mu(B^{\delta_n}(x))}
        = \lim_{n\to\infty} \frac{\mu(B^{\delta_n}(\hat x))}{\mu(B^{\delta_n}(w_n))} 
        \cdot \lim_{n\to\infty} \frac{\mu(B^{\delta_n}(w_n))}
        {\sup_{x \in X} \mu(B^{\delta_n}(x))}
        = 1
    \end{equation*}
    and hence $\hat x$ is a strong mode.

    Conversely, assume that $\hat x$ is a strong mode. Let
    $\{\delta_n\}_{n\in\N}$ be a sequence such that
    $\delta_n \to 0$ and let $\{w_n\}_{n\in\N}$ be any
    approximating sequence. 
    Then 
    \begin{equation*}
        \lim_{n\to\infty} \frac{\mu(B^{\delta_n}(\hat x))}{\mu(B^{\delta_n}(w_n))}
        = \lim_{n\to\infty} \frac{\mu(B^{\delta_n}(\hat x))}
        {\sup_{x \in X} \mu(B^{\delta_n}(x))}
        \cdot \lim_{n\to\infty} \frac{\sup_{x \in X} \mu(B^{\delta_n}(x))}
        {\mu(B^{\delta_n}(w_n))}
        = 1.
    \end{equation*}
    Hence \eqref{eq:property} holds.
\end{proof}

The following example shows that even in the finite-dimensional case, condition~\eqref{eq:property} can be
satisfied for probability measures with discontinuous densities at $\hat
x$.
\begin{example}
    Let $\lambda$ be the distribution of the standard two-dimensional
    Gaussian random variable. Let $\hat x = (0,0)$ be the origin and 
    choose any $\beta > 1$, e.g., $\beta = 2$, and define the cusp 
    \[
        C = \{ (x_1,x_2) \colon x_1 > 0, \abs{x_2} \le x_1^\beta \}.
    \]
    Next we define the probability measure $\mu$ as
    \[
        \mu(A) := \frac{\lambda(A \setminus C)}{\lambda(\R^2 \setminus C)}.
    \]
    Note that $\mu$ has a density which is discontinuous at $\hat x$. However,
    we note two simple facts. First, removing a cusp doesn't change the
    scaled measure much:
    \[
        \mu(B^\delta(\hat x)) = \frac{\lambda(B^\delta(\hat x))}{\lambda(\R^2
        \setminus C)} + O (\delta^{\beta + 1}).
    \]
    Secondly, we have the trivial estimate
    \[
        M^\delta = \sup_{x \in X} \mu(B^\delta(x)) \le \frac{\lambda(B^\delta(\hat x))} {\lambda(\R^2
        \setminus C)}.
    \]
    Therefore, for a certain constant $c > 0$ we have an estimate
    \[
        \frac{\mu(B^\delta(\hat x))}{M^\delta} \ge 1 - c \delta^{\beta -1},
    \]
    since $\lambda(B^{\delta}(\hat x)) = \frac12 \delta^2 + 
    O(\delta^3)$. Taking $\delta \to 0$ thus implies that $\hat x$ is a
    strong mode of $\mu$. Hence by \cref{thm:limit_condition}, the
    condition~\eqref{eq:property} holds.

    We remark that removing a cusp could not be replaced
    by removing a cone to obtain the same example. By a similar
    calculation, it is straightforward to verify that when removing a
    cone instead of a cusp, then regardless of the angle of the cone,
    $\hat x$ is a generalized mode that is not a strong mode for the 
    resulting probability measure.
\end{example}

We can use \cref{thm:limit_condition} to restrict the sets over which the supremum in the definition of $M^\delta$ is taken.
\begin{corollary}
    Let $\hat x \in X$ be a generalized mode of $\mu$. If there exists an $r > 0$ such that
    \begin{equation*}
        \lim_{\delta \to 0} \frac{\mu(B^\delta(\hat x))}{\sup_{w \in B^r(\hat x)} \mu(B^\delta(w))} = 1,
    \end{equation*}
    then $\hat x$ is a strong mode.
\end{corollary}
\begin{proof}
    Let $\{\delta_n\}_{n\in\N}$ be a positive sequence such that $\delta_n \to 0$, 
    and let $\{w_n\}_{n\in\N} \subset X$ be the corresponding approximating sequence.
    Then
    \begin{equation*}
        \lim_{n\to\infty} \frac{\mu(B^{\delta_n}(\hat x))}{\mu(B^{\delta_n}(w_n))}
        \ge \lim_{n\to\infty} \frac{\mu(B^{\delta_n}(\hat x))}
        {\sup_{w \in B^r(\hat x)} \mu(B^{\delta_n}(w))}
        \ge \lim_{\delta \to 0} \frac{\mu(B^\delta(\hat x))}
        {\sup_{w \in B^r(\hat x)} \mu(B^\delta(w))}
        =1
    \end{equation*}
    since $w_n \to \hat x$.
    Hence, \eqref{eq:property} holds, and $\hat x$ is therefore a strong mode by \cref{thm:limit_condition}.
\end{proof}

We illustrate the possibility of satisfying \eqref{eq:property} with the following example.
\begin{example}
    For a probability measure $\mu$ on $\R$ with continuous density $p$ with respect to the Lebesgue measure, \eqref{eq:property} is satisfied in every point $\hat x \in \R$ with $p(\hat x) > 0$.
    To see this, let $\epsilon > 0$. Then there exists a $\delta_0 > 0$ such that for all $x \in B^{\delta_0}(\hat x)$,
    \begin{equation*} 
        \abs{p(x) - p(\hat x)} \le \epsilon. 
    \end{equation*}
    Therefore, for all $\delta \in (0,\frac{\delta_0}{2})$ and $w \in B^{\frac{\delta_0}{2}}(\hat x)$,
    \begin{equation*}
        \bigabs{\frac{\mu(B^\delta(w))}{2\delta} - p(\hat x)} 
        \le \frac{1}{2\delta} \int_{w-\delta}^{w+\delta} \abs{p(x) - p(\hat x)} \di x 
        \le \epsilon.
    \end{equation*}
    So, for every positive sequence $\delta_n \to 0$ and every real sequence $w_n \to \hat x$,
    \begin{equation*}
        \lim_{n\to\infty} \frac{\mu(B^{\delta_n}(w_n))}{2\delta_n} = p(\hat x) > 0.
    \end{equation*}
    This holds true in particular for the constant sequence $w_n = \hat x$, so that
    \begin{equation*}
        \lim_{n\to\infty} \frac{\mu(B^{\delta_n}(\hat x))}{\mu(B^{\delta_n}(w_n))} = \frac{p(\hat x)}{p(\hat x)} = 1.
    \end{equation*}
\end{example}

The property \eqref{eq:property} can be seen as the requirement to have sufficiently fast convergence of the approximating sequences. In other words, we require that the balls $B^\delta(\hat x)$ and $B^{\delta_n}(w_n)$ have asymptotically the same measure. This idea can further quantified by the following theorem.
\begin{theorem} \label{thm:claim1}
    Let $\hat{x} \in X$ be a generalized mode of a Borel probability
    measure $\mu$. If 
    \begin{enumerate}
        \item\label{it:rate} for every positive sequence $\{\delta_n\}_{n\in\N}$ with
            $\delta \to 0$, there exists an approximating sequence $\{w_n\}_{n\in\N}$ such that
            \begin{equation*}
                \lim_{n \to \infty} \frac{\norm{w_n-\hat{x}}_X}{\delta_n} = 0,
            \end{equation*}

        \item\label{it:equicont}
            the family of functions $\{f_n\}_{n\in\N}$ on $[0,1]$ defined by 
            \begin{equation*}
                f_n:[0,1]\to\R, \qquad f_n(r) := \frac{\mu(B^{r(\delta_n + \norm{w_n - \hat x}_X)}(\hat
                x))}{\mu(B^{\delta_n + \norm{w_n - \hat x}_X}(\hat x))},
            \end{equation*}
            is equicontinuous at $r=1$,
    \end{enumerate}
    then $\hat x$ is a strong mode.
\end{theorem}
\begin{proof}
    Notice first that by monotonicity of probability, the function 
    \begin{equation*} 
        \phi_{x,s}(r) := \frac{\mu(B^{rs}(x))}{\mu(B^s(x))}
    \end{equation*}
    is a left-continuous increasing function with right limits on $[0,1]$ for every $x \in X$ and every $s > 0$ and hence so is $f_n= \phi_{\hat x, \delta_n + \norm{w_n - \hat x}_X}$.
    Let now $\{\delta_n\}_{n\in\N}$ be a positive sequence with $\delta_n \to 0$ and let $\{w_n\}_{n\in\N}$ be an approximating sequence satisfying assumption \ref{it:rate}. Setting
    \begin{equation*} 
        r_n := \norm{w_n - \hat{x}}_X \quad \text{for all }n \in \N,
    \end{equation*}
    by the triangle inequality we have
    \begin{equation*}
        \abs{\mu(B^{\delta_n}(\hat x)) - \mu(B^{\delta_n}(w_n))} \le
        \mu(B^{\delta_n+r_n}(\hat x)) - 
        \mu(B^{\delta_n-r_n}(\hat x))
    \end{equation*}
    for $n$ large enough.
    Furthermore, assumption \ref{it:equicont} implies that for every $\epsilon > 0$,
    \begin{equation*}
        f_n(1) -
        f_n\left(\frac{\delta_n - r_n}{\delta_n + r_n}\right) < \epsilon		
    \end{equation*}	  
    for every $n$ large
    enough, since
    \begin{equation*}
        \frac{\delta_n - r_n}{\delta_n + r_n} = \frac{1 - \frac{r_n}{\delta_n}}{1 + \frac{r_n}{\delta_n}} \to 1 \qquad \text{as }n \to \infty. 
    \end{equation*}
    We thus obtain that
    \begin{equation*}
        \limsup_{n \to \infty}\frac{\abs{\mu(B^{\delta_n}(\hat x)) -
        \mu(B^{\delta_n}(w_n))}}{\mu(B^{\delta_n + r_n}(\hat x))} \le \epsilon
    \end{equation*}
    which implies that
    \begin{equation}
        \label{eq:conv_rate_eq1}
        \lim_{n \to \infty}\frac{\mu(B^{\delta_n}(\hat x)) -
        \mu(B^{\delta_n}(w_n))}{\mu(B^{\delta_n + r_n}(\hat x))} 
        =    \lim_{n \to \infty}\frac{\mu(B^{\delta_n}(\hat x))}{\mu(B^{\delta_n + r_n}(\hat x))} \left(1 - \frac{\mu(B^{\delta_n}(w_n))}{\mu(B^{\delta_n}(\hat x))}\right) = 0. 
    \end{equation}
    Since for every $\epsilon > 0$ the assumption \ref{it:equicont} yields $f_n(1) -
    f_n(\delta_n/(\delta_n + r_n)) < \epsilon$ for every $n$ large, it
    also follows that
    \begin{equation*}
        \limsup_{n \to \infty}\frac{\abs{\mu(B^{\delta_n}(\hat x)) -
        \mu(B^{\delta_n + r_n}(\hat x))}}{\mu(B^{\delta_n + r_n}(\hat x))}
        \le \epsilon
    \end{equation*}
    and in particular that
    \begin{equation}
        \label{eq:conv_rate_eq2}
        \lim_{n \to \infty}\frac{\mu(B^{\delta_n}(\hat
        x))}{\mu(B^{\delta_n + r_n}(\hat x))} = 1.
    \end{equation}
    The two limits \eqref{eq:conv_rate_eq1} and \eqref{eq:conv_rate_eq2} together yield that
    \begin{equation*}
        \lim_{n \to \infty}\frac{\mu(B^{\delta_n}(w_n))}{\mu(B^{\delta_n}(\hat x))} = 1
    \end{equation*}
    and, therefore, $\hat x$ is a strong mode by \cref{thm:limit_condition}.
\end{proof}

\begin{remark}
    The equicontinuity condition \ref{it:equicont} is implied by the 
    \emph{$\eta$-annular decay property} ($\eta$-AD) for some $\eta>0$ at the generalized mode
    point (see~\cite{BBL2017} for the definition). In the finite-dimensional
    case, the $1$-AD property at $x$ is equivalent to $f(r) =
    \mu(B^{r}(x))$ being locally absolutely continuous on $(0,\infty)$ and
    $f'(r) r \le c f(r)$ for some $c>0$ and almost every $r > 0$. For example, it can be seen by direct computation that the function $f$ in \cref{ex:standard} has the $1$-AD property at the generalized mode $\hat x = 0$.
\end{remark}

\subsection{Characterization by convergence on a dense subspace}

We next consider the case when the approximating sequences $\{w_n\}_{n\in\N}$ are restricted to a dense subspace of $X$ where a certain continuity of the ratios holds.
\begin{proposition} \label{prop:seqinE}
    Let $\hat{x} \in X$ be a generalized mode of a Borel probability measure $\mu$ and let $E$ be a dense subset of $X$.
    Then for every $\{\delta_n\}_{n\in\N} \subset (0,\infty)$ with $\delta \to 0$ there exists a sequence $\{\tilde{w}_n\}_{n\in\N}\subset E$ with $\tilde{w}_n \to \hat{x}$ in $X$ such that
    \begin{equation*}
        \lim_{n\to\infty} \frac{\mu(B^{\delta_n}(\tilde{w}_n))}{\sup_{x \in X} \mu(B^{\delta_n}(x))} = 1.
    \end{equation*}
\end{proposition}
\begin{proof}
    We first show that for every $\delta > 0$, the mapping $x \mapsto \mu(B^\delta(x))$ is lower semi-continuous.
    To this end, let $\{x_n\}_{n\in\N}\subset X$ be a sequence converging to $x \in X$ and let
    \begin{equation*}
        \chi_A(x) = \begin{cases}
            1, & \text{if }x \in A, \\
            0, & \text{otherwise},
        \end{cases} 
    \end{equation*}
    denote the characteristic function of a set $A \subset X$.
    Then for every $x \in X$, we have 
    \begin{equation*}
        \chi_{B^\delta(x)}(x) \le \liminf_{n\to\infty} \chi_{B^\delta(x_n)}(x), 
    \end{equation*}
    and Fatou's Lemma yields
    \begin{equation*}
        \begin{aligned}[t]
            \mu(B^\delta(x))
            &= \int_X \chi_{B^\delta(x)}(x) \mu(\di x)
            \le \int_X \liminf_{n\to\infty} \chi_{B^\delta(x_n)}(x) \mu(\di x) \\
            &\le \liminf_{n\to\infty} \int_X \chi_{B^\delta(x_n)}(x) \mu(\di x)
            = \liminf_{n\to\infty} \mu(B^\delta(x_n)).
        \end{aligned}
    \end{equation*}

    Now let $\{\delta_n\} \subset (0,\infty)$ with $\delta_n \to 0$ and let $\{w_n\}_{n\in\N} \subset X$ be a corresponding approximating sequence.
    Consider a fixed $n \in \N$.
    By the lower semi-continuity of $x \mapsto \mu(B^{\delta_n}(x))$, we can choose an $R > 0$ such that
    \begin{equation*}
        \mu(B^{\delta_n}(v)) \ge \mu(B^{\delta_n}(w_n)) - \frac{1}{n}M^{\delta_n} \qquad \text{for all }v \in B^R(w_n). 
    \end{equation*}
    Set $r = \min \{R, \frac{1}{n} \}$.
    Because $E$ is dense in $X$, we can choose a $\tilde{w}_n \in B^r(w_n)$, which therefore
    satisfies both
    \begin{equation*}
        1 \ge \frac{\mu(B^{\delta_n}(\tilde{w}_n))}{M^{\delta_n}} \ge \frac{\mu(B^{\delta_n}(w_n))}{M^{\delta_n}} - \frac{1}{n}
    \end{equation*}
    and
    \begin{equation*}
        \norm{\tilde{w}_n - \hat{x}}_X \le \norm{\tilde{w}_n - w_n}_X + \norm{w_n - \hat{x}}_X \le \frac{1}{n} +  \norm{w_n - \hat{x}}_X. 
    \end{equation*}
    As $n \in \N$ was arbitrary, we obtain the desired sequence $\{\tilde{w}_n\}_{n\in\N} \subset E$ with $\tilde{w}_n \to \hat{x}$ in $X$ and
    \begin{equation*}
        1 \ge \lim_{n\to\infty} \frac{\mu(B^{\delta_n}(\tilde{w}_n))}{M^{\delta_n}} 
        \ge \lim_{n\to\infty} \frac{\mu(B^{\delta_n}(w_n))}{M^{\delta_n}} - \lim_{n\to\infty} \frac{1}{n} 
        = 1. 
        \qedhere
    \end{equation*}
\end{proof}

In order to give a sufficient condition for the coincidence of generalized and strong modes, we consider a more specific class of probability measures.
In the following we consider admissible shifts of $\mu$, i.e., elements $h \in X$ such that the shifted measure $\mu_h := \mu(\cdot - h)$ is equivalent to $\mu$.
\enlargethispage{1cm}
\begin{theorem} \label{thm:sufconds1}
    Let $\mu$ be a Borel probability measure with a space of admissible shifts $H$. Suppose that $H$ possesses a dense continuously embedded subspace $(E, \norm{\cdot}_E) \subset H$ 
    such that for every $h\in E$, the density of $\mu_h$ with respect to $\mu$ has a continuous representative $\frac{\di\mu_h}{\di\mu} \in C(X)$. 
    Let $\hat x \in X$ be a generalized mode of $\mu$.  If
    \begin{enumerate}
        \item \label{condconvE} for every $\{\delta_n\}_{n \in \N} \subset (0,\infty)$ with $\delta_n \to 0$ there is an approximating sequence $\{w_n\}_{n \in \N} \subset \hat{x} + E$ with $\norm{w_n - \hat{x}}_E \to 0$, 
        \item \label{condsupcont} there is an $R > 0$ such that 
            \begin{equation*}
                f_R:(E,\norm{\cdot}_E)\to\R,\qquad  f_R(h) := \sup_{x \in B^R(\hat x)} \left|\frac{\di\mu_{h}}{\di\mu}(x) - 1\right|
            \end{equation*}
            is continuous at $0$,
    \end{enumerate}
    then $\hat x$ is a strong mode.
\end{theorem}
\begin{proof}
    Choosing $N$ large enough such that $\delta_n \le R$ for all $n \ge N$, we have for all $n \ge N$ that
    \begin{equation*}
        \begin{aligned}[t]
            \bigabs{\frac{\mu(B^{\delta_n}(w_n))}{\mu(B^{\delta_n}(\hat x))} - 1}
            &= \bigabs{\frac{1}{\mu(B^{\delta_n}(\hat x))} \int_{B^{\delta_n}(\hat x)} \left(\frac{\di\mu_{\hat{x} - w_n}}{\di\mu}(x) - 1\right) \mu(\di x)} \\
            &\le \sup_{x \in B^{\delta_n}(\hat x)} \bigabs{ \frac{\di\mu_{\hat{x} - w_n}}{\di\mu}(x) - 1 } \frac{1}{\mu(B^{\delta_n}(\hat x))} \int_{B^{\delta_n}(\hat x)} \mu(\di x) 
            \le f_R(w_n).
        \end{aligned}
    \end{equation*}
    However, $f_R(w_n) \to f_R(\hat x) = 0$ by the continuity of $f_R$ and the convergence $w_n \to \hat x$ in $E$, so that
    \begin{equation*}
        \lim_{n\to\infty}\frac{\mu(B^{\delta_n}(w_n))}{\mu(B^{\delta_n}(\hat x))} = 1.
    \end{equation*}
    Hence, $\hat x$ is a strong mode by \cref{thm:limit_condition}.
\end{proof}
For a nondegenerate Gaussian measure $\mu = \mathcal{N}(a,Q)$ on a separable Hilbert space $X$ with mean $a \in X$ and covariance operator $Q \in \calL(X)$, the assumptions of \cref{thm:sufconds1} are fulfilled for $E = Q(X)$ and $\hat x = w_n := a$ for all $n \in \N$ as well as any $R>0$. The required continuity properties are also satisfied for, e.g., the Besov measures with $p>1$ discussed in \cite{helin2015maximum}.

\begin{corollary}\label{cor:equistrong}
    Let $\hat x \in X$ be a generalized mode of $\mu$ that satisfies condition \ref{condconvE} of \cref{thm:sufconds1}. If additionally
    \begin{equation*}
        \lim_{w \to_{E} \hat x} \frac{\di\mu_{\hat x - w}}{\di\mu}(\hat x) = 1 
    \end{equation*}
    and there is an $r > 0$ such that the family
    \begin{equation*}
        \left\{ \frac{\di\mu_{\hat x - w}}{\di\mu}: w \in B_E^r(\hat x) \right\}, \quad B_E^r(\hat x) := \left\{ x \in \hat x + E: \norm{x - \hat x}_E < r \right\} \subset \hat x + E, 
    \end{equation*}
    is equicontinuous in $\hat x$, then $\hat x$ is a strong mode
\end{corollary}
\begin{proof}
    We show that condition \ref{condsupcont} of \cref{thm:sufconds1} is satisfied as well, from which the claim then follows.
    For a given $\epsilon > 0$ we choose $0 < r_0 \le r$ small enough such that for all $w \in B_E^{r_0}(\hat x)$ we have that
    \begin{equation*}
        \bigabs{\frac{\di\mu_{\hat x - w}}{\di\mu}(\hat x) - 1} \le \frac{\epsilon}{2}. 
    \end{equation*}
    If we also choose $0 < R \le r_0$ such that for all $x \in B_X^{R}(\hat x)$ and all $w \in B_E^r(w)$ we have that
    \begin{equation*}
        \bigabs{ \frac{\di\mu_{\hat x - w}}{\di\mu}(x) - \frac{\di\mu_{\hat x - w}}{\di\mu}(\hat x)} \le \frac{\epsilon}{2} ,
    \end{equation*}
    then the triangle inequality yields that
    \begin{equation*}
        \abs{f_R(w) - f_R(\hat x)} = \bigabs{\sup_{x \in B^R(\hat x)} \bigabs{\frac{\di\mu_{\hat{x} - w}}{\di\mu}(x) - 1} - 0} \le \epsilon
    \end{equation*}
    for all $w \in B_E^{r_0}(\hat x)$ and therefore the equicontinuity of $f$ at $\hat x$.
\end{proof}

\section{Modes of uniform priors} \label{sec:prior}

We now demonstrate the usefulness of the concept of generalized modes for a class of uniform probability measures on infinite-dimensional spaces that can serve as priors in Bayesian inverse problems.
We first discuss the rigorous construction of the uniform probability measure and then show that such measures admit generalized but in general not strong modes.

\subsection{Construction of the probability measure}

We proceed similar as in \cite{Dashti2017}, with the difference that we define a probability measure on a subspace of $\ell^\infty$ rather than of $L^\infty$. Let us first fix some notation. For $x:=\{x_k\}_{k\in\N} \in \ell^\infty$, we write 
$\norm{x}_\infty = \sup_{k\in\N} \abs{x_k}$.
Furthermore, let $e_j \in \ell^\infty$ for $j \in \N$ denote the standard unit vector in $\ell^\infty$, i.e., $[e_k]_j = 1$ for $j=k$ and $0$ else. 
We then define 
\begin{equation} \label{eq:Xdef}
    X := \overline{\spn\{e_k\}_{k\in\N}} \subset \ell^\infty
\end{equation}
and note that $X = \{x \in \ell^\infty: \lim_{k\to\infty} x_k = 0\}=:c_0$. 
We thus have that $(X,\norm{\cdot}_\infty)$ is a separable Banach space.

We now construct a class of probability measures on $X$ whose mass is concentrated on a set of sequences with strictly bounded components. First, we define a random variable $\xi$ according to the random series
\begin{equation}\label{eq:xidef}
    \xi := \sum_{k=1}^\infty \gamma_k \xi_k e_k,
\end{equation}
where
\begin{enumerate}
    \item $\xi_k \sim \mathcal{U}[-1,1]$ (i.e., uniformly distributed on $[-1,1]$) and
    \item $\gamma_k\geq 0$ for all $k\in \N$ with $\gamma_k\to 0$. 
\end{enumerate}
Note that the partial sums $\xi^n := \sum_{k=1}^n \gamma_k \xi_k e_k$ almost surely form a Cauchy sequence in $X$, since for all $N, m, n \in \N$ with $N \le m \le n$ we have that
\begin{equation*}
    \norm{\xi^n - \xi^m}_\infty
    = \norm{\textstyle\sum_{k=m+1}^n \gamma_k \xi_k e_k}_\infty
    = \sup_{m+1 \le k \le n} \gamma_k \abs{\xi_k}
    \le \sup_{k \ge N} \gamma_k,
\end{equation*}
and the right hand side tends to zero as $N \to \infty$. Since $X$ is complete, the series \eqref{eq:xidef} therefore converges almost surely.
We can thus define the probability measure $\mu_\gamma$ on $X$ by 
\begin{equation}\label{eq:prior_def}
    \mu_\gamma(A) := \Prob{\xi \in A} \quad \text{for every }A \in \mathcal{B}(X),
\end{equation}

The following sets will be important for our study of the generalized and strong modes of $\mu_\gamma$.
We define 
\begin{align*}
    \Eg &:= \{x \in X : \abs{x_k} \le \gamma_k \text{ for all } k \in \N\},
    \intertext{and for every $\delta > 0$,}
    \Egd &:= \{x \in X: \abs{x_k} \le \max \{\gamma_k - \delta, 0\} \text{ for all } k \in \N\} \subset \Eg
    \shortintertext{as well as}
    \Egn &:= \bigcup_{\delta>0} \Egd.
\end{align*}
We first collect some basic properties of $\Eg$ and $\Egd$.
\begin{proposition} \label{prop:Egclosed}
    The sets $\Eg$ and $\Egd$ are convex, compact, and have empty interior.
\end{proposition}
\begin{proof}
    We only consider the case of $\Eg$; the case of $\Egd$ follows analogously.

    First, convexity follows directly from the definition since for every $x, y \in \Eg$ and $\lambda \in [0,1]$ we have that
    \begin{equation*}
        \abs{\{\lambda x + (1 - \lambda) y\}_k} = \abs{\lambda x_k + (1 - \lambda) y_k} \le \lambda \abs{x_k} + (1-\lambda) \abs{y_k} \le \gamma_k\quad\text{for all } k \in \N,
    \end{equation*}
    and hence that $\lambda x + (1 - \lambda) y \in \Eg$.

    For the compactness of $\Eg$, we first show that $X\setminus \Eg$ is open. Let $x \in X \setminus \Eg$ be arbitrary. Then there exists an $m \in \N$ with $\abs{x_m} > \gamma_m$, so that for $\varepsilon := \abs{x_m} - \gamma_m$ we have $B^\varepsilon(x) \in X \setminus \Eg$ as claimed. 
    Hence as a closed subset of a complete space, $\Eg$ is itself complete, which allows us to show compactness by constructing a finite covering of $\Eg$ by balls of radius $\epsilon$. Let $\epsilon > 0$.
    As $\gamma_k\to 0$, there is an $N \in \N$ such that for every $x \in \Eg$ and all $k \ge N+1$,
    \begin{equation*}
        \abs{x_k} \le \gamma_k < \epsilon.
    \end{equation*}
    Now choose $M \in \N$ with $M\epsilon \ge \max \{\gamma_k: k = 1, \dots, N\}$. For every $x \in \Eg$ and $k \in \{1, \dots, N\}$, there is a $\lambda_k \in \{-M,\dots,M\}$ such that $\abs{x_k - \lambda_k\epsilon} < \epsilon$. Hence, $\norm{x - z_N}_\infty < \epsilon$ for $z_N:=\sum_{k=1}^N \lambda_k\epsilon e_k$, which implies that
    \begin{equation*}
        \Eg \subset \bigcup_{k=1}^N B^\epsilon(z_k)
    \end{equation*}
    is the desired finite covering.

    Finally, to show that $\Eg$ has empty interior, let $x \in \Eg$ and $\epsilon > 0$ be arbitrary. We choose $m \in \N$ such that $\gamma_m < \frac{\epsilon}{3}$. Then, $y := x + 3\gamma_m\phi_m \in B^\epsilon(x)$, because $\norm{y - x}_\infty = \norm{3\gamma_m\phi_m}_\infty = 3\gamma_m < \epsilon$, but $y \notin \Eg$. Consequently, $x \in \partial\Eg$ and  hence the interior of $E_\gamma$ is empty.
\end{proof}
Of particular use for finding both generalized and strong modes will be the metric projections onto $\Egd$. For any $\delta >0$, let 
\begin{equation*}
    P^\delta:X\to\Egd,\qquad  P^\delta(x) := \bar x \quad\text{with}\quad \norm{\bar x - x}_\infty = \inf_{z \in \Egd} \norm{z - x}_\infty,
\end{equation*}
which is well-defined because $\Egd$ is closed and convex by \cref{prop:Egclosed}. It is straightforward to show by case distinction that this projection can be characterized componentwise for every $k\in\N$ via
\begin{equation}
    \label{eq:projection}
    [P^\delta(x)]_k =
    \begin{cases}
        0 & \text{if } \gamma_k < \delta,\\
        \gamma_k-\delta & \text{if }x_k > \gamma_k-\delta>0,\\
        x_k & \text{if } x_k \in [-\gamma_k +\delta,\gamma_k-\delta],\\
        -\gamma_k+\delta & \text{if } x_k < -\gamma_k+\delta<0.
    \end{cases}
\end{equation}
The projection satisfies the following properties.
\begin{lemma} \label{lem:projconv}
    Let $x\in \Eg$. Then, $\lim_{\delta\to 0} P^\delta x = x$ in $X$.
\end{lemma}
\begin{proof}
    This follows directly from the definition since
    \begin{equation*}
        \abs{\{P^\delta x\}_k - x_k} \le \delta \qquad\text{for all }k\in \N.
        \qedhere
    \end{equation*}
\end{proof}

Finally, we can give a more explicit characterization of $\Egn$.
\begin{lemma} \label{char_Egn}
    We have that
    \begin{equation*}
        \Egn
        = \left\{x \in X:
            \abs{x_k} < \gamma_k\text{ for all }k \in \N, x_k\neq 0 \text{ for finitely many } k\in\N
        \right\}.
    \end{equation*}
\end{lemma}
\begin{proof}
    By definition, for every $x \in \Egn$ there exists a $\delta > 0$ with $x \in \Egd$
    This implies that $\abs{x_k} \le \max\{\gamma_k - \delta, 0\} < \gamma_k$ for all $k \in \N$.
    Moreover, since $\gamma_k\to 0$ there exists an $N \in \N$ with $\abs{x_k} \le \max\{\gamma_k - \delta, 0\} = 0$ for all $k \ge N$.
    This yields $\gamma_k - \delta \le 0$ for all $k \ge N$.

    Now let $x \in X$ be an element from the set on the right hand side.
    As $\abs{x_k} < \gamma_k$ for all $k \in \N$ and because there are only finitely many $k \in \N$ with $x_k \neq 0$, we can choose
    \begin{equation*} 
        \delta_0 := \min \{ \gamma_k - \abs{x_k}: k \in \N \text{ with }x_k \neq 0 \} > 0
    \end{equation*}
    such that $x \in E_\gamma^{\delta_0} \subset \Egn$.
\end{proof}

\subsection{Small ball probabilities}
We next study the behavior of small balls under the probability measure $\mu_\gamma$, which is crucial for determining modes.
For the sake of conciseness, let 
\begin{equation}
    J_\gamma^\delta(x) := \mu_\gamma(B^\delta(x)) \qquad\text{for all }x \in X,
\end{equation}
for which we have the following straightforward characterization.
\begin{lemma} \label{prop:priorballprob}
    For every $x \in X$ and $\delta > 0$, it holds that
    \begin{equation*}
        \mu_\gamma(\overline{B^\delta(x)}) =  J_\gamma^\delta(x)
        =  \prod_{k=1}^\infty \Prob{\gamma_k \xi_k \in (x_k - \delta, x_k + \delta)},
    \end{equation*}
    where $\overline{A}$ denotes the closure of a set $A \subset X$.
\end{lemma}
\begin{proof}
    First of all, by definition
    \begin{equation*}
        J_\gamma^\delta(x) = \mu_\gamma(B^\delta(x)) = \Prob{\xi \in B^\delta(x)} = \Prob{\abs{\gamma_k \xi_k - x_k} < \delta \text{ for all }k \in \N}. 
    \end{equation*}
    As both $\gamma_k \to 0$ and $x_k \to 0$, there exists an $N \in \N$ such that
    \begin{equation*}
        \gamma_k + \abs{x_k} \le \frac{\delta}{2} \quad \text{for all }k \ge N+1. 
    \end{equation*}
    This implies that
    \begin{equation*}
        \Prob{\abs{\gamma_k \xi_k - x_k} < \delta \text{ for all }k \ge N+1} = 1 
    \end{equation*}
    since $\xi_k \in [-1,1]$ almost surely.
    Consequently,
    \begin{equation*}
        \begin{aligned}[t]
            \mu_\gamma(B^\delta(x))
            &= \Prob{\abs{\gamma_k \xi_k - x_k} < \delta \text{ for }k = 1, \dots, N} 
            \cdot \Prob{\abs{\gamma_k \xi_k - x_k} < \delta \text{ for all }k \ge N+1} \\
            &= \prod_{k=1}^N \Prob{\abs{\gamma_k \xi_k - x_k} < \delta}
        \end{aligned}
    \end{equation*}
    by the independence of the $\xi_k$. This yields the second identity.

    The first identity now follows from 
    \begin{equation*}
        \Prob{\gamma_k \xi_k \in [x_k - \delta, x_k + \delta]} 
        = \Prob{\gamma_k \xi_k \in (x_k - \delta, x_k + \delta)} \quad \text{for all }k \in \N. \qedhere 
    \end{equation*}
\end{proof}

We can use this characterization to show that for every $\delta>0$, the origin maximizes $J_\gamma^\delta$.
\begin{proposition} \label{prop:maxprobzero} 
    Let $\delta>0$. Then 
    \begin{equation*}
        J_\gamma^\delta(0) = \max_{x \in X} J_\gamma^\delta(x) > 0.
    \end{equation*}
\end{proposition}
\begin{proof}
    From the second equality in \cref{prop:priorballprob}, we have that
    \begin{equation*}
        \mu_\gamma(B^\delta(x))
        = \prod_{k=1}^\infty \Prob{\gamma_k \xi_k \in (x_k - \delta, x_k + \delta)}
        \le \prod_{k=1}^\infty \Prob{\gamma_k \xi_k \in (-\delta, \delta)} 
        = \mu_\gamma(B^\delta(0))
    \end{equation*}
    for every $x = \{x_k\}_{k\in\N} \in X$. On the other hand, $\mu_\gamma(B^\delta(0)) \le \sup_{x \in X} \mu_\gamma(B^\delta(x))$, which shows that $x=0$ is a maximizer.

    For the positivity, note that as $\gamma_k \to 0$, there exists an $N \in \N$ such that
    \begin{equation*}
        \mu_\gamma(B^\delta(0)) = \prod_{k=1}^N \Prob{\gamma_k \xi_k \in (-\delta,\delta)}>0
    \end{equation*}
    since $\Prob{\gamma_k \xi_k \in (-\delta,\delta)} > 0$ for all $k \leq N$.
\end{proof}

Crucially, $J_\gamma^\delta$ is constant on $\Egd$.
\begin{proposition} \label{lem:projconst}
    For every $\delta > 0$, 
    \begin{equation*}
        J_\gamma^\delta(x_1)=J_\gamma^\delta(x_2) \qquad\text{for all }x_1,x_2\in \Egd.
    \end{equation*}
\end{proposition}
\begin{proof}
    Let $x_1, x_2 \in \Egd$. For every $k \in \N$, we distinguish between the following two cases:
    \begin{enumerate}
        \item $\delta \le \gamma_k$: In this case,
            \begin{equation*}
                (x_{1,k} - \delta, x_{1,k} + \delta) \subseteq (-\gamma_k, \gamma_k) \quad \text{and} \quad (x_{2,k} - \delta, x_{2,k} + \delta) \subseteq (-\gamma_k, \gamma_k),
            \end{equation*}
            so that
            \begin{equation*}
                \Prob{\gamma_k \xi_k \in (x_{1,k} - \delta, x_{1,k} + \delta)}
                = \frac{\delta}{\gamma_k}
                = \Prob{\gamma_k \xi_k \in (x_{2,k} - \delta, x_{2,k} + \delta)}.
            \end{equation*}
        \item $\delta > \gamma_k$: In this case, $x_{1,k} = x_{2,k} = 0$ by definition of $\Egd$, and hence
            \begin{equation*}
                \Prob{\gamma_k \xi_k \in (x_{1,k} - \delta, x_{1,k} + \delta)}
                = 1
                = \Prob{\gamma_k \xi_k \in (x_{2,k} - \delta, x_{2,k} + \delta)}.
            \end{equation*}
    \end{enumerate}
    Together we obtain that
    \begin{equation*}
        \begin{aligned}[t]
            J_\gamma^\delta(x_1) &= \Prob{\xi \in B^\delta(x_1)}
            = \prod_{k=1}^\infty \Prob{\gamma_k \xi_k \in (x_{1,k} - \delta, x_{1,k} + \delta)} \\
            &= \prod_{k=1}^\infty \Prob{\gamma_k \xi_k \in (x_{2,k} - \delta, x_{2,k} + \delta)}
            = \Prob{\xi \in B^\delta(x_2)} = J_\gamma^\delta(x_2))
        \end{aligned}
    \end{equation*}
    as claimed.
\end{proof}
Since $0\in \Egd$,
combining \cref{prop:maxprobzero,lem:projconst} immediately yields that \emph{every} $x\in \Egd$ maximizes $J_\gamma^\delta$. 
\begin{corollary} \label{cor:Egdmaxprob}
    For $\delta > 0$ and every $\bar x \in \Egd$,
    \begin{equation*}
        J_\gamma^\delta(\bar x) = \max_{x \in X} J_\gamma^\delta(x) > 0.
    \end{equation*}
\end{corollary}

The following proposition and its corollary will be useful in computing and estimating ratios $J_\gamma^\delta(x)/J_\gamma^\delta(0)$ of small ball probabilities.
\begin{proposition} \label{comp_prob_expl}
    Let $x \in \Eg$, $k \in \N$ and $\delta > 0$. If $\gamma_k > 0$ we have
    \begin{equation*}
        \frac{\Prob{\gamma_k\xi_k \in (x_k - \delta, x_k + \delta)}}{\Prob{\gamma_k\xi_k \in (-\delta, \delta)}} = 
        \begin{cases}
            1, & \text{if }\delta \le \gamma_k - \abs{x_k}, \\
            \frac{\delta + \gamma_k - \abs{x_k}}{2\delta}, & \text{if }\delta \in (\gamma_k - \abs{x_k}, \gamma_k], \\
            \frac{\delta + \gamma_k - \abs{x_k}}{2\gamma_k}, & \text{if }\delta \in (\gamma_k, \gamma_k + \abs{x_k}) \\
            1, & \text{if }\delta \ge \gamma_k + \abs{x_k}.
        \end{cases}
    \end{equation*}
    If $\gamma_k = 0$, on the other hand, we have
    \begin{equation*}
        \frac{\Prob{\gamma_k\xi_k \in (x_k - \delta, x_k + \delta)}}{\Prob{\gamma_k\xi_k \in (-\delta, \delta)}} = 1.
    \end{equation*}
\end{proposition}
\begin{proof}
    First of all, $\abs{x_k} \le \gamma_k$ by definition of $\Eg$.
    Hence, if $\gamma_k = 0$ then also $x_k = 0$, so that
    \[ \frac{\Prob{\gamma_k\xi_k \in (x_k - \delta, x_k + \delta)}}{\Prob{\gamma_k\xi_k \in (-\delta, \delta)}} = 1 \]
    in this case.
    Now assume that $\gamma_k > 0$.
    If $\delta \le \gamma_k$, then
    \[ \Prob{\gamma_k\xi_k \in (-\delta, \delta)} = \frac{2\delta}{2\gamma_k} = \frac{\delta}{\gamma_k}, \]
    whereas $\Prob{\gamma_k\xi_k \in (-\delta, \delta)} = 1$ if $\delta > \gamma_k$.

    In case $\delta \le \gamma_k - \abs{x_k}$ we find that
    \[ {\Prob{\gamma_k\xi_k \in (x_k - \delta, x_k + \delta)}} = \frac{2\delta}{2\gamma_k} = \frac{\delta}{\gamma_k}, \]
    and in case $\delta \ge \gamma_k + \abs{x_k}$ we find that
    \[ {\Prob{\gamma_k\xi_k \in (x_k - \delta, x_k + \delta)}} = 1. \]
    In the remaining case $\gamma_k - \abs{x_k} < \delta \le \gamma_k + \abs{x_k}$ we compute
    \[ {\Prob{\gamma_k\xi_k \in (x_k - \delta, x_k + \delta)}} = \frac{\delta + (\gamma_k - \abs{x_k})}{2\gamma_k}. \]
    The proposition now follows from combining these expressions.
\end{proof}

\cref{comp_prob_expl} yields a lower bound which is independent of $\delta$.
\begin{corollary} \label{ineq_comp_prob}
    Let $x \in \Eg$, $k \in \N$. If $\gamma_k > 0$ then the inequality
    \begin{equation*}
        \frac{\Prob{\gamma_k\xi_k \in (x_k - \delta, x_k + \delta)}}{\Prob{\gamma_k\xi_k \in (-\delta, \delta)}}
        \ge 1 - \frac{\abs{x_k}}{2\gamma_k}
    \end{equation*}
    holds for all $\delta > 0$, where we have equality for $\delta = \gamma_k$.
\end{corollary}
\begin{proof}
    We use \cref{comp_prob_expl}.
    Clearly, $\Prob{\gamma_k\xi_k \in (x_k - \delta, x_k + \delta)}/\Prob{\gamma_k\xi_k \in (-\delta, \delta)}$ attains its minimum in $\delta = \gamma_k$, in which case we have
    \begin{equation*}
        \frac{\Prob{\gamma_k\xi_k \in (x_k - \delta, x_k + \delta)}}{\Prob{\gamma_k\xi_k \in (-\delta, \delta)}}
        = \frac{2\gamma_k - \abs{x_k}}{2\gamma_k}. \qedhere
    \end{equation*}
\end{proof}

\subsection{Modes of the probability measure}

Finally, we characterize both the strong and the generalized modes of the uniform probability measure and show that these do not coincide.
\begin{theorem}
    Every point $\hat{x} \in \Egn$ is a strong mode of $\mu_\gamma$.
\end{theorem}
\begin{proof}
    By definition of $\Egn$ there is a $\delta > 0$ such that $\hat{x} \in \Egd$.
    Then 
    \begin{equation*}
        M^\delta = \max_{x \in X} \mu_\gamma(B^\delta(x)) = \mu_\gamma(B^\delta(\hat{x})) 
    \end{equation*}
    by \cref{cor:Egdmaxprob}, and therefore
    \begin{equation*}
        \lim_{\delta \to 0} \frac{\mu_\gamma(B^\delta(\hat{x}))}{M^\delta} = \lim_{\delta \to 0}  \frac{\mu_\gamma(B^\delta(\hat{x}))}{\mu_\gamma(B^\delta(\hat{x}))} = 1. \qedhere 
    \end{equation*}
\end{proof}
The following example shows that there may also be strong modes outside of $\Egn$.
\begin{example}
    We choose $\gamma_k$, $k \in \N$, and $x \in \Eg$ in such a way that for any $\delta > 0$, 
    \[
        \frac{\Prob{\gamma_k\xi_k \in (x_k - \delta, x_k + \delta)}}{\Prob{\gamma_k\xi_k \in (-\delta, \delta)}} = 1 
    \]
    for all components except for one.
    For this purpose, set $\gamma_k := \frac{1}{k(k+2)}$ and $x_k := \frac{1}{k+1}\gamma_k$ for all $k \in \N$.
    Then, $\gamma_k\to 0$,
    \[ 
        \frac{\abs{x_k}}{\gamma_k} = \frac{1}{k+1} \to 0 \quad \text{as }k \to \infty,
    \]
    and
    \begin{align*}
        \gamma_{k+1} + \abs{x_{k+1}} &= \gamma_{k+1}\left(1 + \frac{1}{k+2}\right) = \frac{1}{(k+1)(k+3)}\cdot\frac{k+3}{k+2} \\
        &= \frac{1}{k(k+2)}\cdot\frac{k}{k+1} = \gamma_k\left(1 - \frac{1}{k+1}\right) = \gamma_k - \abs{x_k}
    \end{align*}
    for all $k \in \N$.
    In particular, $x \notin \Egn$ by \cref{char_Egn} because all of its components are different from zero.

    Now for given $\delta > 0$ we choose $m = m(\delta) \in \N$ such that $\gamma_m - \abs{x_m} \le \delta < \gamma_m + \abs{x_m}$.
    Then 
    \[ \frac{\Prob{\gamma_k\xi_k \in (x_k - \delta, x_k + \delta)}}{\Prob{\gamma_k\xi_k \in (-\delta, \delta)}} = 1 \]
    for all $k \in \N \setminus \{m\}$ by \cref{comp_prob_expl} and
    \[ \frac{\mu_\gamma(B^\delta(x))}{\mu_\gamma(B^\delta(0))}
        = \frac{\Prob{\gamma_m\xi_m \in (x_m - \delta, x_m + \delta)}}{\Prob{\gamma_m\xi_m \in (-\delta, \delta)}}
    \ge 1 - \frac{\abs{x_m}}{2\gamma_m} = 1 - \frac{1}{2(m + 1)} \]
    by \cref{prop:priorballprob,ineq_comp_prob}.
    Moreover, $m = m(\delta) \to \infty$ as $\delta \to 0$.
    Consequently,
    \[ \lim_{\delta \to 0} \frac{\mu_\gamma(B^\delta(x))}{M^\delta}
    = \lim_{\delta \to 0} \frac{\mu_\gamma(B^\delta(x))}{\mu_\gamma(B^\delta(0))} = 1 \]
    by \cref{prop:maxprobzero}. Hence $x$ is a strong mode of $\mu_\gamma$.
\end{example}

However, the following results show that \cref{def:mode} is too restrictive in this case and in fact can be unintuitive.
\begin{proposition} \label{prop:strongmodes}
    The following claims hold:
    \begin{enumerate}
        \item\label{it:strongmodes1} If for $x \in \Eg$ there is an $m\in \N$ with $\abs{x_m} = \gamma_m>0$, then $x$ is not a strong mode of $\mu_\gamma$.
        \item\label{it:strongmodes2} There are $\gamma, x \in X$ with $\abs{x_k} < \gamma_k$ for all $k \in \N$ such that $x$ is not a strong mode of $\mu_\gamma$.
    \end{enumerate}
\end{proposition}
\begin{proof}
    \emph{Ad \ref{it:strongmodes1}:}
    First, note that for $\delta > 0$ small enough, 
    \begin{equation*} 
        \Prob{\gamma_m \xi_m \in (x_m - \delta, x_m + \delta)} = \frac12 \Prob{\gamma_m \xi_m \in (-\delta, \delta)}. 
    \end{equation*}
    Moreover, 
    \begin{equation*}
        \Prob{\gamma_k \xi_k \in (x_k - \delta, x_k + \delta)} \le \Prob{\gamma_k \xi_k \in (-\delta, \delta)}
    \end{equation*}
    for all $k \in \N$ and $\delta > 0$, so that
    \begin{equation*}
        \frac{\mu_\gamma(B^\delta(x))}{\mu_\gamma(B^\delta(0))}
        = \prod_{k=1}^\infty \frac{\Prob{\gamma_k \xi_k \in (x_k - \delta, x_k + \delta)}}
        {\Prob{\gamma_k \xi_k \in (-\delta, \delta)}}
        \le \frac12
    \end{equation*}
    for $\delta > 0$ small enough.
    But by \cref{prop:maxprobzero}, we also have that
    \begin{equation*}
        \limsup_{\delta \to 0} \frac{\mu_\gamma(B^\delta(x))}{M^\delta}
        = \limsup_{\delta \to 0} \frac{\mu_\gamma(B^\delta(x))}{\mu_\gamma(B^\delta(0))} \le \frac12 < 1.
    \end{equation*}
    Hence, $x$ cannot be a strong mode.

    \emph{Ad \ref{it:strongmodes2}:}
    We take $\gamma_k = \frac{1}{k}$ and $x_k = \frac12\gamma_k = \frac{1}{2k}$ for all $k \in \N$. For given $\delta \in (0, \frac14)$ we also choose $n \in \N$ such that $\frac{1}{n} \ge \delta > \frac{1}{n+1}$, i.e., $\frac12\gamma_n< \delta \leq \gamma_n$. Hence,
    \begin{align*}
        \Prob{\gamma_n \xi_n \in (x_n - \delta, x_n + \delta)} &= \frac{\gamma_n - (x_n - \delta)}{2\gamma_n}
        \le \frac{\gamma_n - (-\frac12\gamma_n)}{2\gamma_n} = \frac34,
        \shortintertext{as well as}
        \Prob{\gamma_n \xi_n \in (-\delta,\delta)} &= \frac{\delta}{\gamma_n} > \frac{n}{n+1} \ge \frac{4}{5}
    \end{align*}
    for $n \ge 4$.
    In addition,
    \begin{equation*}
        \Prob{\gamma_k x_k \in (x_k - \delta, x_k + \delta)} \le \Prob{\gamma_k x_k \in (-\delta,\delta)} \qquad \text{for all }k \in \N.
    \end{equation*}
    Consequently,
    \begin{equation*}
        \begin{aligned}[t]
            \frac{\mu_\gamma(B^\delta(x))}{\mu_\gamma(B^\delta(0))}
            &= \prod_{k=1}^\infty \frac{\Prob{\gamma_k \xi_k \in (x_k - \delta, x_k + \delta)}}
            {\Prob{\gamma_k \xi_k \in (-\delta, \delta)}} \\
            &\le \frac{\Prob{\gamma_n \xi_n \in (x_n - \delta, x_n + \delta)}}
            {\Prob{\gamma_n \xi_n \in (-\delta, \delta)}}
            \le \frac{3}{4}\cdot\frac{5}{4} = \frac{15}{16}.
        \end{aligned}
    \end{equation*}
    \Cref{prop:maxprobzero} now yields that
    \begin{equation*}
        \limsup_{\delta \to 0} \frac{\mu_\gamma(B^\delta(x))}{M^\delta}
        = \limsup_{\delta \to 0} \frac{\mu_\gamma(B^\delta(x))}{\mu_\gamma(B^\delta(0))} \le \frac{15}{16} < 1. 
    \end{equation*}
    and hence that $x$ is not a strong mode.
\end{proof}

On the other hand, every point in $\Eg$ is a generalized mode and vice versa.
\begin{theorem} \label{thm:gmodesprior}
    A point $x \in X$ is a generalized mode of $\mu_\gamma$ if and only if $x \in \Eg$.
\end{theorem}
\begin{proof}
    Assume first that $x \in \Eg$ and set $w^\delta := P^\delta x \in \Egd$ for all $\delta > 0$. Then, $w^\delta \to x$ as $\delta \to 0$ by \cref{lem:projconv} and $\mu_\gamma(B^\delta(w^\delta)) = \mu_\gamma(B^\delta(0))$ by \cref{lem:projconst}, so that
    \begin{equation*}
        \frac{\mu_\gamma(B^\delta(w^\delta))}{\mu_\gamma(B^\delta(0))} = 1\quad \text{for all }\delta > 0.
    \end{equation*}
    Since  $M^\delta := \max_{x \in X} \mu_\gamma(B^\delta(x)) = \mu_\gamma(B^\delta(0))$ by \cref{prop:maxprobzero}, it follows from \cref{prop:curvecond} that $x$ is a generalized mode.

    Conversely, assume that $x \in X \setminus \Eg$. Then there exists an $m \in \N$ with $\abs{x_m} > \gamma_m$. Taking now $\delta_0 := \abs{x_m} - \gamma_m > 0$, we have that
    \begin{equation*}
        \Prob{\gamma_m \xi_m \in (x_m - \delta, x_m + \delta)} = 0 \qquad \text{for all }\delta \in (0,\delta_0) 
    \end{equation*}
    and hence that
    \begin{equation*}
        \mu_\gamma(B^\delta(x))
        = \prod_{k=1}^\infty \Prob{\gamma_k \xi_k \in (x_k - \delta, x_k + \delta)} = 0 \quad \text{for all }\delta \in (0, \delta_0).
    \end{equation*}
    This implies that 
    \begin{equation*}
        \frac{\mu_\gamma(B^\delta(w))}{M^\delta} = 0 \quad \text{for all }\delta \in \left(0, \frac{\delta_0}{2}\right)\text{ and }w \in B^{\frac{\delta_0}{2}}(x), 
    \end{equation*}
    and hence $x$ cannot be a generalized mode.
\end{proof}

\section{Variational characterization of generalized MAP estimates}
\label{sec:MAP}

In this section, we consider Bayesian inverse problems with uniform priors and characterize the corresponding generalized MAP estimates -- i.e., the generalized modes of the posterior distribution -- as minimizers of an appropriate objective functional. 

Let $X$ be the separable Banach space defined by \eqref{eq:Xdef} and choose as prior $\mu_0$ the uniform probability distribution $\mu_\gamma$ defined by \eqref{eq:prior_def} for some non-negative sequence of weights $\gamma_k\to0$.
We assume that for given data $y$ from a Banach space $Y$ the posterior distribution $\mu^y$ satisfies $\mu^y \ll \mu_0$ and can be expressed as
\begin{equation}
    \label{eq:posterior}
    \mu^y(A) = \frac{1}{Z(y)} \int_{A} \exp(-\Phi(x;y))\mu_0(dx)
\end{equation}
for all $A \in \mathcal{B}(X)$, where
\begin{equation*}
    Z(y) := \int_X \exp(-\Phi(x;y))\mu_0(dx)
\end{equation*}
is a positive and finite normalization constant and $\Phi$: $X \times Y \to \R$ is the likelihood. This way, $\mu^y$ constitutes a probability measure on $X$.
Throughout this section, we fix $y \in Y$ and abbreviate $\Phi(x;y)$ by $\Phi(x)$.
We make the following assumption on the likelihood.
\begin{assumption} \label{ass:likelihood}
    The function $\Phi: X \to \R$ is Lipschitz continuous on bounded sets, i.e., for every $r>0$, there exists $L=L_r> 0$ such that for all $x_1,x_2 \in X$ with $\norm{x_1}_X, \norm{x_2}_X \le r$ we have
    \begin{equation*}
        \abs{\Phi(x_1)-\Phi(x_2)} \leq L \norm{x_1-x_2}_X.
    \end{equation*}
\end{assumption}

In \cref{thm:gmodesprior}, we have seen that the indicator function $\iota_\Eg$ of $\Eg$ in the sense of convex analysis, i.e.,
\begin{equation*}
    \iota_\Eg:X\to\Rbar:= \R\cup\{\infty\},
    \qquad \iota_\Eg(x) := 
    \begin{cases}
        0,	& \text{if }x \in \Eg, \\
        \infty,	& \text{otherwise}.
    \end{cases}
\end{equation*}
is the suitable functional to minimize in order to find the generalized modes of the prior measure $\mu_0$.
In contrast, the common Onsager--Machlup functional is not defined for $\mu_0$, as $\mu_0$ is not quasi-invariant with respect to shifts along any direction $x \in X \setminus \{0\}$. The goal of this section is to relate a similar functional that characterizes generalized MAP estimates to a suitably generalized limit of small ball probabilities.
Specifically, we define the functional
\begin{equation}\label{eq:onsagermachlup}
    I:X\to\Rbar,\qquad  I(x) := \Phi(x) + \iota_\Eg(x) = 
    \begin{cases}
        \Phi(x), & \text{if }x \in \Eg, \\
        \infty, & \text{otherwise}.
    \end{cases}
\end{equation}
Furthermore, for $\delta>0$ we denote by
\begin{equation*}
    J^\delta(x) := \mu^y(B^\delta(x)) \qquad\text{for all }x \in X 
\end{equation*}
the small ball probabilities of the posterior measure. Let $x^\delta$ denote a (not necessarily unique) maximizer of $J^\delta$ in $\Egd$, i.e.,
\begin{equation*}
    J(x^\delta) = \max_{x \in \Egd} J^\delta(x).
\end{equation*}
What we prove below (in \cref{thm:mainthm}) is the following:
\begin{enumerate}
    \item $\{x^\delta\}_{\delta > 0}$ contains a convergent subsequence (because $\Eg$ is compact), and the limit of the convergent subsequence minimizes $I$.
    \item Any minimizer of $I$ is a generalized MAP estimate and vice versa.
\end{enumerate}
We first demonstrate the necessity to work with subsequences by an example which shows that the sequence $\{x^\delta\}_{\delta>0}$ may not have a unique limit.
\begin{example}\label{ex:cluster}
    Define $g(0) = 0$,
    \begin{equation*}
        g(x) = \frac34 \frac1{2^n}, \quad \text{ if} \quad \frac1{2^{n+1}}
        < \abs{x} \le \left(\frac12\right)^n, \quad n \in \Z 
    \end{equation*}
    and consider a probability measure $\mu$ on $\R$ with density $f$ with respect to the Lebesgue measure, defined via 
    \begin{equation*}
        \tilde f(x) := \max \left\{ 1 - g(x - 1), 1 - \sqrt{2}g\left(\frac{x + 1}{\sqrt{2}}\right), 0 \right\}, \qquad f(x) := \frac{\tilde f(x)}{\int_{\R} \tilde f(x)\,dx}.
    \end{equation*}
    To find the set of points $x$ for which the probability of $B^\delta(x)$ is maximal under $\mu$ for given $\delta\in (0,1)$, first let $n\in \N$ be such that $\frac1{2^{n+1}}< \delta \leq \frac1{2^n}$. By a case distinction, one can then verify that
    \begin{equation*}
        \argmax_{x \in \R} \mu(B^\delta(x)) =
        \begin{cases}
            \left[-1 - \frac{\sqrt{2}}{2^{n+1}} + \delta, -1 + \frac{\sqrt{2}}{2^{n+1}} - \delta\right]
            & \text{if } \delta \in \left(\frac{1}{2^{n+1}},\frac{\sqrt{2}}{2^{n+1}}\right],\\[1ex] 
            \left[1 - \frac1{2^n} + \delta, 1 + \frac1{2^n} - \delta\right]
            & \text{if } \delta \in \left(\frac{\sqrt{2}}{2^{n+1}},\frac1{2^{n}}\right],
        \end{cases}
    \end{equation*}
    Hence, any family $\{x^\delta\}_{\delta > 0}$ of maximizers $x^\delta \in \argmax_{x \in \R} \mu(B^\delta(x))$ has the two cluster points $-1$ and $1$; see \cref{fig:example_cluster}.
\end{example}
\begin{figure}
    \centering
    \begin{tikzpicture}

\begin{axis}[%
width=0.8\textwidth,
height=0.2\textheight,
scale only axis,
xmin=-2.5,
xmax=2.5,
ymin=-0.1,
ymax=1.2,
]
\addplot [color=DarkBlue,line width=1.0pt]
  table[row sep=crcr]{%
-2.5	0\\
-1.7096048024012	0\\
-1.70710355177589	0.46966991411009\\
-1.3544272136068	0.46966991411009\\
-1.35192596298149	0.734834957055045\\
-1.1768384192096	0.734834957055045\\
-1.17433716858429	0.867417478527523\\
-1.08929464732366	0.867417478527523\\
-1.08679339669835	0.933708739263761\\
-1.04427213606803	0.933708739263761\\
-1.04177088544272	0.966854369631881\\
-1.02426213106553	0.966854369631881\\
-1.02176088044022	0.98342718481594\\
-1.01175587793897	0.98342718481594\\
-1.00925462731366	0.99171359240797\\
-1.00675337668834	0.99171359240797\\
-1.00425212606303	0.995856796203985\\
-1.00175087543772	0.997928398101993\\
-0.999249624812407	0.998964199050997\\
-0.994247123561781	0.99171359240797\\
-0.989244622311156	0.99171359240797\\
-0.986743371685843	0.98342718481594\\
-0.979239619809905	0.98342718481594\\
-0.976738369184592	0.966854369631881\\
-0.956728364182091	0.966854369631881\\
-0.954227113556779	0.933708739263761\\
-0.911705852926463	0.933708739263761\\
-0.909204602301151	0.867417478527523\\
-0.82416208104052	0.867417478527523\\
-0.821660830415208	0.734834957055045\\
-0.646573286643322	0.734834957055045\\
-0.644072036018009	0.46966991411009\\
-0.293896948474237	0.46966991411009\\
-0.291395697848924	0\\
-0.00125062531265652	0\\
0.00125062531265652	0.25\\
0.498999499749875	0.25\\
0.501500750375187	0.625\\
0.74912456228114	0.625\\
0.751625812906453	0.8125\\
0.874187093546773	0.8125\\
0.876688344172086	0.90625\\
0.93671835917959	0.90625\\
0.939219609804903	0.953125\\
0.966733366683342	0.953125\\
0.969234617308654	0.9765625\\
0.98424212106053	0.9765625\\
0.986743371685843	0.98828125\\
0.991745872936468	0.98828125\\
0.994247123561781	0.994140625\\
0.999249624812406	0.999267578125\\
1.00175087543772	0.99853515625\\
1.00425212606303	0.994140625\\
1.00675337668834	0.994140625\\
1.00925462731366	0.98828125\\
1.01425712856428	0.98828125\\
1.01675837918959	0.9765625\\
1.02926463231616	0.9765625\\
1.03176588294147	0.953125\\
1.06178089044522	0.953125\\
1.06428214107054	0.90625\\
1.12431215607804	0.90625\\
1.12681340670335	0.8125\\
1.24937468734367	0.8125\\
1.25187593796898	0.625\\
1.49949974987494	0.625\\
1.50200100050025	0.25\\
1.99974987493747	0.25\\
2.00225112556278	0\\
2.5	0\\
};

\end{axis}
\end{tikzpicture}%
    \caption{probability density $\tilde f$ constructed in \cref{ex:cluster} for which the family of maximizers $\{x^\delta\}_{\delta>0}$ has two cluster points $\hat x_1=-1$ and $\hat x_2=1$}\label{fig:example_cluster}
\end{figure}
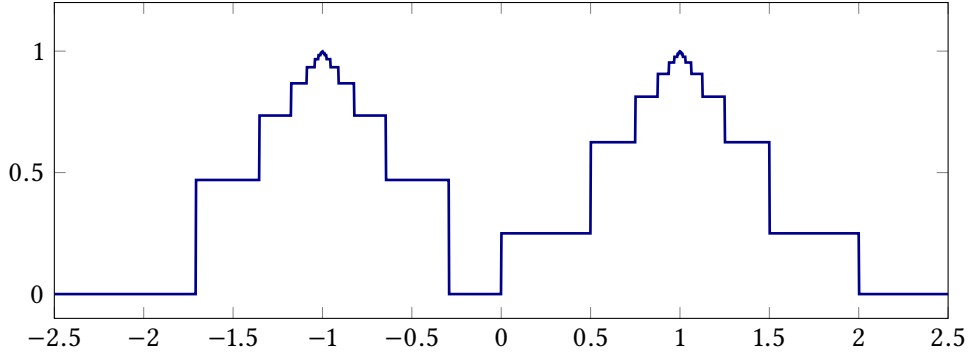

We begin our analysis by collecting some auxiliary results on $\Phi$ and $I$.
\begin{lemma} \label{lem:phibounded}
    If \cref{ass:likelihood} holds, then there exists a constant $M>0$ such that
    \begin{equation}
        \abs{\Phi(x)} \leq M \qquad\text{for all }x\in \Eg.
    \end{equation}
\end{lemma}
\begin{proof}
    Since $\Eg$ is compact and therefore bounded by \cref{prop:Egclosed}, there exists some $R\geq 0$ such that $\norm{x}_\infty\leq R$ for all $x\in \Eg$. The Lipschitz continuity of $\Phi$ on bounded sets then implies that for any $x\in \Eg$,
    \begin{equation*}
        \abs{\Phi(x)} 
        \le \abs{\Phi(x) - \Phi(0)} + \abs{\Phi(0)}
        \le L\norm{x}_\infty + \abs{\Phi(0)}
        \le LR + \abs{\Phi(0)}=:M.
        \qedhere
    \end{equation*}
\end{proof}
\begin{lemma} \label{lem:exmin}
    If \cref{ass:likelihood} holds, then there exists an $\bar x \in \Eg$ such that
    \begin{equation*}
        I(\bar x) = \min_{x\in X} I(x).
    \end{equation*}
\end{lemma}
\begin{proof}
    As $I(x) = \infty$ for all $x \in X \setminus \Eg$, we only need to consider $I|_\Eg = \Phi|_\Eg$, which is Lipschitz continuous on the bounded set $\Eg$ by \cref{ass:likelihood}.
    Furthermore, the set $\Eg$ is nonempty and by \cref{prop:Egclosed} closed and compact.
    The Weierstrass Theorem therefore implies that $I$ attains its minimum in a point $\bar x \in \Eg$.
\end{proof}

We also need the following results on the small ball probabilities of the posterior. 
\begin{lemma} \label{lem:Jnotzero}
    Assume that \cref{ass:likelihood} holds. Then for every $\delta>0$ there exists a constant $c(\delta) > 0$ such that 
    \begin{equation*}
        J^\delta(x) \ge c(\delta) >0\qquad \text{for all }x \in \Egd. 
    \end{equation*}
\end{lemma}
\begin{proof}
    First, $\mu_0(B^\delta(x)) = \mu_0(B^\delta(0)) > 0$ for all $x \in \Egd$ by \cref{lem:projconst,cor:Egdmaxprob}.
    Hence we can estimate using \cref{lem:phibounded} that
    \begin{equation*}
        J^\delta(x)
        =  \frac 1{Z(y)} \int_{B^\delta(x)} \exp(-\Phi(x)) \mu_0(\di x)
        \ge \frac{1}{Z(y)}e^{-M}\mu_0(B^\delta(0))
        =: c(\delta) > 0,
    \end{equation*}
    because $Z(y)$ is positive and finite by assumption.
\end{proof}
We next show show that $J^\delta$ is maximized within $E^\delta_\gamma$ by making use of the metric projection $P^\delta : X \to E^\delta_\gamma$ defined in \eqref{eq:projection}.
\begin{lemma} \label{lem:projinc}
    For all $\delta > 0$, we have that
    \begin{equation}
        J^\delta(P^\delta x) \ge J^\delta(x)\qquad\text{for all }x\in X.
    \end{equation}
\end{lemma}
\begin{proof}
    We note that $B^\delta(x) \setminus \overline{B^\delta(P^\delta x)} \subset X \setminus \Eg$, so that $\mu_0(B^\delta(x) \setminus \overline{B^\delta(P^\delta x)}) = 0$.
    Moreover, $\mu_0(\overline{B^\delta(P^\delta x)} \setminus B^\delta(P^\delta x)) = 0$ by \cref{prop:priorballprob}, so that $\mu_0(B^\delta(x) \setminus B^\delta(P^\delta x)) = 0$ as well.
    With this knowledge we estimate
    \begin{equation*}
        \begin{aligned}[b]
            J^\delta(x) 
            &
            = \frac 1{Z(y)} \int_{B^\delta(x)} \exp(-\Phi(x)) \mu_0(\di x) \\
            &\le \frac 1{Z(y)} \int_{B^\delta(x) \cup B^\delta(P^\delta x)} 
            \exp(-\Phi(x)) \mu_0(\di x) \\
            &= \frac 1{Z(y)} \int_{B^\delta(P^\delta x)} \exp(-\Phi(x)) \mu_0(\di x)
            = J^\delta(P^\delta x).
        \end{aligned} 
        \qedhere
    \end{equation*}
\end{proof}
\begin{corollary}
    For all $\delta>0$,
    \begin{equation*}
        J(x^\delta) := \max_{x \in \Egd} J^\delta(x) =  \max_{x \in X} J^\delta(x).
    \end{equation*}
\end{corollary}
\begin{proof}
    If $x \in X$ is a maximizer of $J^\delta$, then $P^\delta x \in \Egd$ is a maximizer as well.
\end{proof}

The following proposition indicates that the functional $I$ plays the role of a generalized Onsager--Machlup functional for the posterior distribution $\mu^y$ associated with the uniform prior $\mu_0$ defined by \eqref{eq:prior_def}.
\begin{proposition} \label{lem:postballprob}
    Suppose that \cref{ass:likelihood} holds.
    Let $x_1, x_2 \in \Eg$ and $\{w_1^\delta\}_{\delta > 0}, \{w_2^\delta\}_{\delta > 0} \subset \Eg$ with
    \begin{enumerate}
        \item\label{it:postballprob1} $w_1^\delta, w_2^\delta \in \Egd$ for all $\delta > 0$,
        \item\label{it:postballprob2} $w_1^\delta \to x_1$ and $w_2^\delta \to x_2$ as $\delta \to 0$.
    \end{enumerate}
    Then,
    \begin{equation} \label{eq:OMproperty}
        \lim_{\delta \to 0} \frac{J^\delta(w_1^\delta)}{J^\delta(w_2^\delta)} 
        = \exp(I(x_2) - I(x_1)).
    \end{equation}
\end{proposition}
\begin{proof}
    Let $\delta > 0$. We consider
    \begin{equation*}
        \begin{aligned}[t]
            \frac{J^\delta(w_1^\delta)}{J^\delta(w_2^\delta)}
            &= \frac{\int_{B^\delta(w_1^\delta)} \exp(-\Phi(x)) \mu_0(\di x)}
            {\int_{B^\delta(w_2^\delta)} \exp(-\Phi(x)) \mu_0(\di x)} \\
            &= \exp(\Phi(x_2) - \Phi(x_1)) \frac
            {\int_{B^\delta(w_1^\delta)} \exp(\Phi(x_1) - \Phi(x)) \mu_0(\di x)}
            {\int_{B^\delta(w_2^\delta)} \exp(\Phi(x_2) - \Phi(x)) \mu_0(\di x)},
        \end{aligned}
    \end{equation*}
    which is well-defined by \cref{lem:Jnotzero}.
    By \cref{ass:likelihood},
    \begin{align*}
        \abs{\Phi(x_1) - \Phi(x)} &\le L\norm{x_1 - x}_\infty \le L\left(\norm{x_1 - w_1^\delta}_\infty + \norm{w_1^\delta - x}_\infty\right) \le L\left(\norm{x_1 - w_1^\delta}_\infty + \delta\right)
        \intertext{for all $x \in B^\delta(w_1^\delta)$ and}
        \abs{\Phi(x_2) - \Phi(x)} &\le L\norm{x_2 - x}_\infty \le L\left(\norm{x_2 - w_2^\delta}_\infty + \delta\right)
    \end{align*}
    for all $x \in B^\delta(w_2^\delta)$, where $L$ is the Lipschitz constant of $\Phi$ on the bounded set $\bigcup_{x\in \Eg}B^\delta(x)$.
    It follows that
    \begin{multline*}
        \exp(\Phi(x_2) - \Phi(x_1)) e^{-L\left(\norm{x_1 - w_1^\delta}_\infty + \norm{x_2 - w_2^\delta}_\infty + 2\delta\right)} \frac{\mu_0(B^\delta(w_1^\delta))}{\mu_0(B^\delta(w_2^\delta))}
        \le \frac{J^\delta(w_1^\delta)}{J^\delta(w_2^\delta)} 
        \\
        \le \exp(\Phi(x_2) - \Phi(x_1)) e^{L\left(\norm{x_1 - w_1^\delta}_\infty + \norm{x_2 - w_2^\delta}_\infty + 2\delta\right)} \frac{\mu_0(B^\delta(w_1^\delta))}{\mu_0(B^\delta(w_2^\delta))}.
    \end{multline*}
    Now \cref{lem:projconst} implies that $\mu_0(B^\delta(w_1^\delta)) = \mu_0(B^\delta(w_2^\delta))$ for all $\delta>0$, and hence taking the limit $\delta\to 0$ and using \ref{it:postballprob2} together with the fact that $I(x)=\Phi(x)$ for all $x\in \Eg$ yields the claim.
\end{proof}
Note that \cref{lem:projconv} immediately implies that \eqref{eq:OMproperty} also holds true for the prior distribution $\mu_0$ if $I$ is replaced by $\iota_\Eg$.
However, it is an open question how \cref{lem:postballprob} can be generalized to cover a wider class of distributions while sustaining the connection to generalized MAP estimates described in the following theorem, which is the main result of this section.
\begin{theorem} \label{thm:mainthm}
    Suppose \cref{ass:likelihood} holds. Then the following hold:
    \begin{enumerate}
        \item\label{it:mainthm:exist} For every $\{\delta_n\}_{n\in\N} \subset (0,\infty)$, the sequence $\{x^{\delta_n}\}_{n\in\N}$ contains a subsequence that converges strongly in $X$ to some $\bar{x} \in \Eg$.
        \item\label{it:mainthm:min} Any cluster point $\bar{x}\in\Eg$ of $\{x^{\delta_n}\}_{n\in\N}$ is a minimizer of $I$.
        \item\label{it:mainthm:MAP} A point $\hat{x} \in X$ is a minimizer of $I$ if and only if it is a generalized MAP estimate for $\mu^y$.
    \end{enumerate}
\end{theorem}
\begin{proof} 
    \emph{Ad \ref{it:mainthm:exist}:}
    By definition, $x^\delta \in \Egd \subset \Eg$ for all $\delta > 0$. Since $\Eg$ is compact and closed by \cref{prop:Egclosed}, there exists a convergent subsequence, again denoted by $\{x^{\delta_n}\}_{n\in\N}$, with limit $\bar{x} \in \Eg$.

    \emph{Ad \ref{it:mainthm:min}:}
    Let now $\bar x\in \Eg$ be the limit of an arbitrary convergent subsequence -- still denoted by $\{x^{\delta_n}\}_{n\in\N}$ -- and assume that $\bar{x}$ is not a minimizer of $I$. 
    By \cref{lem:exmin}, a minimizer $x^* \in \Eg$ of $I$ exists and by assumption satisfies $I(\bar{x}) > I(x^*)$.
    Moreover, $P^\delta x^* \to x^*$ as $\delta \to 0$ by \cref{lem:projconv}.
    Now the definition of $x^\delta$ and \cref{lem:postballprob} yield
    \begin{equation*}
        1 \le \lim_{n \to \infty} \frac{\mu^y(B^{\delta_n}(x^{\delta_n}))}{\mu^y(B^{\delta_n}(P^{\delta_n} x^*))} = \exp(I(x^*) - I(\bar{x})) < 1,
    \end{equation*}
    a contradiction. So $\bar{x}$ is in fact a minimizer of $I$.

    \emph{Ad \ref{it:mainthm:MAP}:}
    Now let $x^* \in X$ be a minimizer of $I$ which by definition satisfies $x^*\in \Eg$.
    To see that $x^*$ is a generalized MAP estimate, we choose $w^\delta := P^\delta {x^*} \in \Egd$ for every $\delta>0$, which implies that $\norm{w^\delta - x^*}_\infty \le \delta$.
    Furthermore, by definition of $x^\delta$ and $M^\delta$,  
    \begin{equation*}
        \mu^y(B^\delta(x^\delta)) = \max_{x \in X} \mu^y(B^\delta(x)) = M^\delta \qquad\text{for all }\delta>0.
    \end{equation*}
    By \cref{ass:likelihood} and the boundedness of $\bigcup_{x\in\Eg}B^\delta(x)$, we can use the Lipschitz continuity of $\Phi$ to obtain the estimate
    \begin{equation*}
        \Phi(x) \le \Phi(x^*) + L\left(\norm{x - w^\delta}_\infty + \norm{w^\delta - x^*}_\infty\right) \le \Phi(x^*) + 2L\delta \quad \text{for all }x \in B^\delta(w^\delta). 
    \end{equation*}
    On the other hand, a similar argument shows that
    \begin{equation*}
        \Phi(x) \ge \Phi(x^\delta) - L\norm{x - x^\delta}_\infty \ge \Phi(x^\delta) - L\delta \quad \text{for all }x \in B^\delta(x^\delta). 
    \end{equation*}
    Consequently,
    \begin{equation*}
        \begin{aligned}[t]
            \mu^y(B^\delta(w^\delta))
            &= \frac{1}{Z(y)} \int_{B^\delta(w^\delta)} \exp(-\Phi(x)) \mu_0(\di x) \\
            &\ge \exp(-\Phi(x^*) - 2L\delta) \frac{1}{Z(y)} \mu_0(B^\delta(w^\delta)),
        \end{aligned}
    \end{equation*}
    and
    \begin{equation*}
        \begin{aligned}[t]
            M^\delta &= \mu^y(B^\delta(x^\delta))
            = \frac{1}{Z(y)} \int_{B^\delta(x^\delta)} \exp(-\Phi(x)) \mu_0(\di x) \\
            &\le \exp(-\Phi(x^\delta) + L\delta) \frac{1}{Z(y)} \mu_0(B^\delta(x^\delta)).
        \end{aligned}
    \end{equation*}
    As $\mu_0(B^\delta(w^\delta) = \mu_0(B^\delta(x^\delta))$ by \cref{lem:projconst}, combining these two estimates yields
    \begin{equation*}
        \frac{\mu^y(B^{\delta}(w^\delta))}{\mu^y(B^\delta(x^\delta))}
        \ge \exp(\Phi(x^\delta) - \Phi(x^*) - 3L\delta).
    \end{equation*}
    Now the minimizing property of $x^*$ implies that
    \begin{equation*}
        \lim_{\delta \to 0} \frac{\mu^y(B^{\delta}(w^\delta))}{\mu^y(B^\delta(x^\delta))}
        \ge \lim_{\delta \to 0} \exp(-3L\delta) = 1
    \end{equation*}
    and hence that $x^*$ is a generalized MAP estimate.

    Conversely, let $\hat{x}$ be a generalized MAP estimate and $\{\delta_n\} \subset (0,\infty)$ such that $\delta_n \to 0$ with corresponding approximating sequence $\{w_n\}_{n\in\N} \subset X$, i.e., $w_n \to \hat{x}$ and
    \begin{equation*}
        \lim_{n\to\infty} \frac{\mu^y(B^{\delta_n}(w_n))}{M^{\delta_n}} = 1. 
    \end{equation*}
    From \cref{lem:projconv}, it follows that $P^{\delta_n}w_n \to \hat x$ as well.
    We then have $\hat{x} \in \Eg$, because otherwise the closedness of $\Eg$ by \cref{prop:Egclosed} would imply that $\mu^y(w_n) = 0$ for $n$ large enough.
    Also, by definition of $x^\delta$, 
    \begin{equation*}
        \mu^y(B^\delta(x^\delta)) = \max_{x \in X} \mu^y(B^\delta(x)) = M^\delta 
    \end{equation*}
    for all $\delta > 0$.  
    Now by \ref{it:mainthm:exist} and \ref{it:mainthm:min}, we may extract a subsequence, again denoted by $\{x^{\delta_n}\}_{n\in\N}$, such that $x^{\delta_n}\to \bar{x} \in \Eg$ and $\bar x$ is a minimizer of $I$.	
    \Cref{lem:postballprob,lem:projinc} then yield that
    \begin{equation*}
        \exp(I(\bar{x}) - I(\hat{x})) 
        = \lim_{n \to \infty} \frac{\mu^y(B^{\delta_n}(P^{\delta_n} w_n))}{\mu^y(B^{\delta_n}(x^{\delta_n}))} 
        \ge \lim_{n \to \infty} \frac{\mu^y(B^{\delta_n}(w_n))}{\mu^y(B^{\delta_n}(x^{\delta_n}))} = 1.
    \end{equation*}
    Hence, $I(\bar{x}) \ge I(\hat{x})$, i.e., $\hat{x}$ is a minimizer of $I$.
\end{proof}

\begin{remark}\label{rem:ivanov}
    \Cref{thm:mainthm} shows that for inverse problems subject to Gaussian noise as in the following section, the generalized MAP estimate coincides with Ivanov regularization with the specific choice of the compact set $\Eg$; compare \eqref{eq:ivanov_functional} below with, e.g., \cite{IvaVasTan02,LorWor13,NeuRam14}. Hence, Ivanov regularization can be considered as a non-parametric MAP estimate for a suitable choice of the compact set the solution is restricted to. 

    Furthermore, we point out that for linear inverse problems where the Ivanov functional is convex, the minimizers generically lie on the boundary of the compact set; see, e.g., \cite[Prop.~2.2\,(iii)]{NeuRam14} and \cite[Cor.~2.6]{CK18}. However, this is not possible for strong modes due to \cref{prop:strongmodes}\,(i), which further indicates the need to consider generalized modes in this setting.
\end{remark}

\begin{remark}
    Following up on \cref{rem:ivanov}, we briefly remark on using the variational characterization \eqref{eq:ivanov_functional} for the computation of the generalized MAP estimate. Assume first for simplicity that $\Phi(x) = \frac12\norm{x-z}_2^2$ for given $z\in \ell^2$. Proceeding as in \cite{CK18}, one can then use classical tools from convex analysis \cite{NAO:2017} to show that the generalized MAP estimate $\hat x$ has the componentwise representation
    \begin{equation}\label{eq:map_denoise}
        \hat x_k = \proj_{[-\gamma_k,\gamma_k]}(z_k) := 
        \begin{cases} 
            \gamma_k  & \text{if } z_k > \gamma_k,\\
            z_k  & \text{if } |z_k| \leq \gamma_k,\\
            -\gamma_k & \text{if } z_k <-\gamma_k.
        \end{cases}
    \end{equation}
    In particular, any finite-dimensional MAP estimate $\hat x^N$ obtained by truncating the random series \eqref{eq:xidef} at $k=N$ coincides with the infinite-dimensional MAP estimate $\hat x$ up to this index.

    If $\Phi(x) = \frac12\norm{F(x)-z}_2^2$ for some (possibly nonlinear but Fréchet-differentiable) forward operator $F:\ell^2\to\ell^2$, \eqref{eq:map_denoise} becomes
    \begin{equation*}
        \hat x_k = \proj_{[-\gamma_k,\gamma_k]}\left(\hat x_k - [F'(\hat x)^*(F'(\hat x)-z)]_k\right),
    \end{equation*}
    which is a Lipschitz continuous fixed point equation from $\ell^2$ to $\ell^2$ for $\hat x$ that can be solved by either a fixed point iteration (i.e., forward--backward splitting similar to (F)ISTA) or a semismooth Newton method \cite{NAO:2017}. Again -- assuming a suitable discretization for $F$ -- this componentwise characterization can be used to show convergence of finite-dimensional MAP estimates $\hat x^N \to \hat x$ as $N\to \infty$. 
\end{remark}

\section{Consistency in nonlinear inverse problems with Gaussian noise}
\label{sec:consistency}

Finally, we show consistency in the small noise limit of the generalized MAP estimate from \cref{sec:MAP}. For the sake of presentation, we focus here on the special case of finite-dimensional data and additive Gaussian noise.

Let $Y = \R^K$ for a fixed $K \in \N$, endowed with the Euclidean norm, and $F:X\to Y$ be a closed nonlinear operator.
Assuming $x \sim \mu_0$ for the prior measure $\mu_0$ defined via \eqref{eq:prior_def} and $y\in Y$ is a corresponding measurement of $F(x)$ corrupted by additive Gaussian noise with mean $0$ and positive definite covariance operator $\Sigma \in \R^{K \times K}$ scaled by $\delta>0$, the corresponding posterior measure $\mu^y$ is $\mathcal{N}(0,\Sigma)$-almost surely given by
\begin{equation*}
    \frac{\di \mu^y}{\di \mu_0}(x) = \frac{1}{Z(y)} \exp(-\Phi(x;y))\quad\text{for $\mu_0$-almost all }x\in X,
\end{equation*}
where
\begin{equation*}
    \Phi(x;y) := \frac{1}{2\delta^2} \norm{ \Sigma^{-\frac12}(F(x) - y) }_Y^2
\end{equation*}
for all $x \in X$ and $y \in Y$.
Furthermore, we assume that for every $R > 0$, the restriction of $F$ to $B^R(0)$ is Lipschitz continuous, so that \cref{ass:likelihood} is satisfied. 
Then \cref{thm:mainthm} yields that the generalized MAP estimates for $\mu^y$ are $\mathcal{N}(0,\Sigma)$-almost surely given by the minimizers of the functional
\begin{equation}\label{eq:ivanov_functional}
    I(x) = \frac{1}{2\delta^2} \norm{ \Sigma^{-\frac12}(F(x) - y) }_Y^2 + \iota_\Eg(x).
\end{equation}

Now let $\{\delta_n\}_{n\in\N} \subset (0,\infty)$ with $\delta_n \to 0$ and consider a frequentist setup where we have a true solution $x^\dagger \in \Eg$ and a sequence $\{y_n\}_{n\in\N} \in Y$ of measurements given by
\begin{equation*}
    y_n = F(x^\dagger) + \delta_n \eta_n,
\end{equation*}
where $\eta_n \sim \mathcal{N}(0, \Sigma)$ are independently and identically distributed.
Moreover, define
\begin{equation}\label{eq:functional_n}
    I_n(x) := \frac{1}{2\delta_n^2} \norm{ \Sigma^{-\frac12}(F(x) - y_n) }_Y^2 + \iota_\Eg(x)
\end{equation}
for all $n \in \N$ and $x \in X$. Let $x_n \in \Eg$ denote a minimizer of $I_n$ for all $n \in \N$. 
Then we have the following consistency result.
\begin{theorem} \label{thm:conssmallnoise}
    Suppose that $F$: $X \to Y$ is closed and $x^\dagger \in \Eg$. Then any sequence $\{x_n\}_{n\in\N}$ of minimizers of  \eqref{eq:functional_n} contains a convergent subsequence whose limit $\bar{x} \in \Eg$ satisfies $F(\bar{x}) = F(x^\dagger)$ almost surely.
\end{theorem}
\begin{proof}
    Let $e_k^*(x) := x_k$ for all $x \in X$ and $k \in \N$. As $x_n \in \Eg$ for all $n \in \N$, the sequence $\{x_n\}_{n\in\N}$ is almost surely bounded with 
    \begin{equation*}
        \abs{\dualpair{e_k^*,x_n}_X} =\abs{[x_n]_k} \le \gamma_k \quad\text{almost surely for all } k \in \N. 
    \end{equation*}
    We can thus extract for every $k\in \N$ a subsequence with $\Exp{\dualpair{e_k^*,x_n}_X} \to \bar{x}_k$ and $\abs{\bar{x}_k} \le \gamma_k$. Here and in the following, we pass to the subsequence without adapting the notation. By a diagonal sequence argument, we can thus construct a sequence, again denoted by $\{x_n\}_{n\in\N}$, satisfying
    \begin{equation*}
        \Exp{\dualpair{e_k^*,x_n}_X} \to \bar{x}_k \quad \text{for all }k \in \N 
    \end{equation*}
    and $\bar{x} := \sum_{k\in\N} \bar{x}_k e_k \in \Eg$.
    We now write for every $v \in X^* \cong \ell^1$ and $N \in \N$
    \begin{equation*}
        \begin{aligned}[t]
            \bigabs{\Exp{\dualpair{v,x_n - \bar{x}}_X}} 
            &= \bigabs{\Exp{\textstyle\sum_{k = 1}^\infty \dualpair{v,e_k}_X \dualpair{e_k^*,x_n - \bar{x}}_X}} \\
            &\le \sum_{k=1}^N \abs{\dualpair{v,e_k}_X} \bigabs{\Exp{\dualpair{e_k^*,x_n - \bar{x}}_X}} + \sum_{k=N+1}^\infty \abs{\dualpair{v,e_k}_X} 2\gamma_k \\
            &\le \norm{v}_{X^*} \sup_{k \in \{1,\dots,N\}} \bigabs{\Exp{\dualpair{e_k^*,x_n - \bar{x}}_X}}
            + 2\norm{v}_{X^*} \sup_{k \ge N+1} \gamma_k.
        \end{aligned}
    \end{equation*}
    Since $\gamma_k\to$, we can for every $\epsilon > 0$ choose $N$ large enough such that 
    \begin{equation*}
        2\norm{v}_{X^*} \sup_{k \ge N+1} \gamma_k \le \frac{\epsilon}{2}, 
    \end{equation*}
    and then choose $n_0$ large enough such that also
    \begin{equation*}
        \norm{v}_{X^*} \bigabs{\Exp{\dualpair{e_k^*,x_n - \bar{x}}_X}} \le \frac{\epsilon}{2} \quad \text{for all } n \ge n_0 \quad\text{and } k \in \{1,\dots,N\} 
    \end{equation*}
    since $\Exp{\dualpair{e_k^*,x_n}_X} \to \dualpair{e_k^*,\bar{x}}_X$.
    Consequently, 
    \begin{equation*}
        \abs{\Exp{\dualpair{v,x_n - \bar{x}}_X}} \le \epsilon \qquad\text{for all } n \ge n_0,
    \end{equation*}
    i.e., $\Exp{\dualpair{v,x_n - \bar{x}}_X} \to 0$. 
    From this we conclude that $\dualpair{v,x_n - \bar{x}}_X \to 0$ in probability as $n \to \infty$ for all $v \in X^*$. 
    This implies the existence of a subsequence that converges weakly to $\bar{x}$ almost surely.
    Therefore, by compactness of $\Eg$, a further subsequence converges strongly to $\bar{x}$ almost surely.

    Furthermore, inserting the definition of $y_n$ yields that
    \begin{equation*}
        \begin{aligned}[t]
            I_n(x) &= \frac{1}{2\delta_n^2} \norm{ \Sigma^{-\frac12}(F(x^\dagger) - F(x) + \delta_n\eta_n) }_Y^2 \\
            &= \frac{1}{2\delta_n^2} \norm{ \Sigma^{-\frac12}(F(x^\dagger) - F(x)) }_Y^2
            + \frac{1}{\delta_n} \dualpair{ \Sigma^{-\frac12}(F(x^\dagger) - F(x)), \Sigma^{-\frac12}\eta_j}_Y + \frac12 \sum_{j=1}^n \norm{ \Sigma^{-\frac12}\eta_j }_Y^2
        \end{aligned}
    \end{equation*}
    for all $x \in \Eg$. Note that the first two terms vanish for $x=x^\dag\in\Eg$ and that $I_n(x_n)\leq I_n(x^\dag)$ by definition of $x_n$. Rearranging this inequality and using Cauchy's and Young's inequalities, we thus obtain that
    \begin{equation*}
        \begin{aligned}[t]
            \norm{ \Sigma^{-\frac12}(F(x^\dagger) - F(x_n)) }_Y^2
            &\le - 2\delta_n \dualpair{ \Sigma^{-\frac12}(F(x^\dagger) - F(x_n)), \Sigma^{-\frac12}\eta_j}_Y \\
            &\le 2\delta_n \norm{\Sigma^{-\frac12}(F(x^\dagger) - F(x_n))}_Y \norm{\Sigma^{-\frac12}\eta_j}_Y \\
            &\le \frac12 \norm{\Sigma^{-\frac12}(F(x^\dagger) - F(x_n))}_Y^2 + 2\delta_n^2 \norm{\Sigma^{-\frac12}\eta_j}_Y^2.
        \end{aligned}
    \end{equation*}
    Consequently,
    \begin{equation*}
        \norm{ \Sigma^{-\frac12}(F(x_n) - F(x^\dagger)) }_Y^2 \le 4\delta_n^2 \norm{\Sigma^{-\frac12}\eta_j}_Y^2, 
    \end{equation*}
    and hence
    \begin{equation*}
        \Exp{\norm{ \Sigma^{-\frac12}(F(x_n) - F(x^\dagger)) }_Y^2} \le 4\delta_n^2 \Exp{\norm{\Sigma^{-\frac12}\eta_j}_Y^2} = 4K\delta_n^2\to 0, 
    \end{equation*}
    which also implies that
    \begin{equation*}
        \norm{ \Sigma^{-\frac12}(F(x_n) - F(x^\dagger)) }_Y \to 0 \quad\text{in probability.}
    \end{equation*}
    After passing to a further subsequence, we thus obtain that $F(x_n)\to F(x^\dagger)$ almost surely.
    The closedness of $F$ then yields that $F(\bar{x}) = F(x^\dagger)$ almost surely.
\end{proof}
Note that although \cref{thm:conssmallnoise} does not require the Lipschitz continuity of $F$ on bounded sets, in general without it we do not know if the minimizers of $I_n$ exist or if they are generalized MAP estimates for $\mu^y$.

By a subsequence-subsequence argument, we can obtain convergence in probability of the full sequence under the usual assumption on $F$.
\begin{corollary} \label{cor:consFinj}
    If $F$ is injective, then $x_n \to x^\dagger$ in probability as $n \to \infty$.
\end{corollary}
\begin{proof}
    We can apply the proof of \cref{thm:conssmallnoise} to every subsequence of $\{x_n\}_{n\in\N}$ to obtain a further subsequence that converges to $\bar{x} = x^\dagger$ almost surely. This in turn is equivalent to the convergence of the whole sequence to $x^\dagger$ in probability by \cite[Cor.~6.13]{Klenke:2014}.
\end{proof}

\begin{remark}
    Consistency in the large sample size limit can be shown analogously. 
    Specifically, assuming a fixed noise level $\delta>0$ (which without loss of generality we can fix at $\delta=1$) and $n$ independent measurements $y_1, \dots, y_n \in Y$, the posterior measure $\mu^n$ is $\mathcal{N}(0,\Sigma)^n$-almost surely given by
    \begin{equation*}
        \frac{\di \mu^n}{\di \mu_0}(x) 
        := \left(\prod_{j=1}^n \frac{1}{Z(y_j)}\right) 
        \exp \left( -\frac12 \sum_{j=1}^n \norm{ \Sigma^{-\frac12}(F(x) - y_j) }_Y^2 \right)\quad\text{for $\mu_0$-almost all } x\in X.
    \end{equation*}
    Again we consider a frequentist setup with a true solution $x^\dagger \in \Eg$ and
    \begin{equation*}
        y_j = F(x^\dagger) + \eta_j \quad \text{for all }j \in \N.
    \end{equation*}
    Then it can be shown as in \cref{thm:conssmallnoise} that any sequence of minimizers $x_n \in \Eg$ of
    \begin{equation*}
        I_n(x) := \frac12 \sum_{j=1}^n \norm{ \Sigma^{-\frac12}(F(x) - y_j) }_Y^2 + \iota_\Eg(x)
    \end{equation*}
    contains a convergent subsequence whose limit $\bar{x} \in \Eg$ satisfies $F(\bar{x}) = F(x^\dagger)$ almost surely and that the full sequence converges to $x^\dagger$ in probability if $F$ is injective.
\end{remark}

\section{Conclusion}

We have proposed a novel definition of generalized modes and corresponding MAP estimates for non-parametric Bayesian statistics. Our approach extends the construction of \cite{dashti2013map} by replacing the fixed base point with an approximating sequence in the convergence of small ball probabilities. This allows covering cases where the previous approach fails, such as that of priors having discontinuous densities. Our definition coincides with the definition from \cite{dashti2013map} for Gaussian priors as well as for priors which admit an approximating sequence satisfying additional convergence properties. For uniform priors defined via a random series with bounded coefficients, such generalized modes can be shown to exist. Furthermore, the corresponding generalized MAP estimates in Bayesian inverse problems (i.e., generalized modes of the posterior distribution) can be characterized as minimizers of a functional that can be seen as a generalization of the Onsager--Machlup functional in the Gaussian case. For inverse problems with finite-dimensional Gaussian noise, the generalized MAP estimates are consistent in the small noise as well as in the large sample limit. This result can be extended to other noise models for which the likelihood satisfies \cref{ass:likelihood} by adapting the proof of \cref{thm:conssmallnoise} appropriately.

This work can be extended in several further directions. Clearly, the practical behavior of the generalized MAP estimate for the uniform prior, including its discretization and numerical solution, should be studied for relevant inverse problems. Regarding theoretical issues, the first question is whether generalized MAP estimates arise as $\Gamma$-limits of finite-dimensional MAP estimates in the sense that approximating sequences can be constructed from MAP estimates for a sequence of finite-dimensional problems. (Note that since MAP estimates may lie on the boundary of $E_\gamma$, see \cref{rem:ivanov}, these these can in general not simply be direct discretizations of the infinite-dimensional MAP estimate.)
It would furthermore be of interest to study whether a generalized Onsager--Machlup functional similar to \eqref{eq:onsagermachlup} can be derived for generalized MAP estimates of a wider class of priors than the ones considered in \cref{sec:prior,sec:MAP}. 
Of particular practical relevance would be Gaussian priors with non-negativity constraints, where our technique does not directly apply due to the lack of a series representation such as \eqref{eq:uniform}.
Finally, it is an open question whether our generalized MAP estimate can be interpreted as a proper Bayesian estimator, since prior work on this topic \cite{burgerlucka2014,pereyra2016} does not apply to discontinuous distributions in infinite dimensions.

\section*{Acknowledgments}
The authors would like to thank the anonymous reviewers for helpful and constructive suggestions on the presentation. 
TH was supported by the Academy of Finland via project 275177. PP was
partially funded by Finnish Centre of Excellence in Inverse Problems
(project 284715) and has been partially supported by Finnish Centre of Excellence in Inverse Modelling and Imaging (project 312119).

\bibliographystyle{jnsao}
\bibliography{generalizedmodes.bib}

\end{document}